\documentclass{amsart}
\usepackage[english]{babel}
\usepackage{amsmath, amsthm, amssymb}
\usepackage{enumerate}
\usepackage{geometry}
\usepackage{hyperref}
\usepackage{cleveref}

\newcommand{\R}{\mathbb{R}}
\newcommand{\Q}{\mathbb{Q}}
\newcommand{\Z}{\mathbb{Z}}

\renewcommand{\O}{\mathcal{O}}
\renewcommand{\P}{\mathcal{P}}
\newcommand{\calC}{\mathcal{C}}

\newcommand{\calQ}{\mathcal{Q}}
\newcommand{\frakC}{\mathfrak{c}}
\DeclareMathOperator{\Reg}{Reg}

\DeclareMathOperator{\Nm}{N}

\newcommand{\leg}[2]{\left( \frac{#1}{#2} \right)}
\newcommand{\Prim}{\mathrm{Prim}}

\theoremstyle{plain}
\newtheorem{theorem}{Theorem}[section]
\newtheorem{prop}[theorem]{Proposition}
\newtheorem{lemma}[theorem]{Lemma}

\theoremstyle{definition}

\title{Counting principal ideals of small norm in the simplest cubic fields}
\author{Mikul\'a\v{s} Zindulka}
\address{Charles University, Faculty of Mathematics and Physics, Department of Algebra,
Sokolovsk\'{a} 83, 186 75 Praha 8, Czech Republic}
\email{mikulas.zindulka@matfyz.cuni.cz}
\thanks{The author acknowledges support by Czech Science Foundation (GA\v{C}R) grant 21-00420M, Charles University project PRIMUS/24/SCI/010, and GAUK project No. 134824.}

\begin{document}

\begin{abstract}
We estimate the number of principal ideals $ I $ of norm $ \mathrm{N}(I) \leq x $ in the family of the simplest cubic fields. The advantage of our result is that it provides the correct order of magnitude for arbitrary $ x \geq 1 $, even when $ x $ is significantly smaller than the discriminant. In particular, it shows that there exist surprisingly many principal ideals of small norm.
\end{abstract}

\maketitle

\section{Introduction}
\label{sec:Intro}

Let $ K $ be a number field of degree $ d $ and let $ I_K(x) $ be the number of integral ideals $ I $ in $ K $ with norm $ \Nm(I) \leq x $. The function $ I_K $ is called the \emph{ideal counting function}. It is a classical result of Weber~\cite{We61} that
\begin{equation}
\label{eq:IW}
	I_K(x) = \kappa_K\cdot x+O_K\left(x^{1-\frac{1}{d}}\right),
\end{equation}
where
\[
	\kappa_K = \frac{2^{r_1}(2\pi)^{r_2}\Reg_K h_K}{w\sqrt{|\Delta_K|}}.
\]
The constant $ \kappa_K $ depends on the invariants of the number field: $ r_1 $ is the number of real embeddings, $ 2r_2 $ is the number of complex embeddings, $ \Reg_K $ is the regulator of $ K $, $ h_K $ is the class number, $ w $ is the number of roots of unity contained in $ K $, and $ \Delta_K $ is the discriminant.

Weber's result was improved by Landau \cite[Satz 210]{La49} to
\begin{equation}
\label{eq:IL}
	I_K(x) = \kappa_K\cdot x+O_K\left(x^{1-\frac{2}{d+1}}\right).
\end{equation}

A closely related problem is to find an asymptotic formula for the number of integral ideals $ I $ in $ K $ with norm $ \Nm(I) \leq x $ which belong to a given class $ \frakC $ of the class group. If $ P_K(x,\frakC) $ denotes the counting function for these ideals, then
\begin{equation}
\label{eq:P}
	P_K(x,\frakC) = C_K\cdot x+O_K\left(x^{1-\frac{1}{d}}\right),
\end{equation}
where
\[
	C_K = \frac{\kappa_K}{h_K} = \frac{2^{r_1}(2\pi)^{r_2}\Reg_K}{w\sqrt{|\Delta_K|}}.
\]
When $ \frakC $ is the class of principal ideals, the \emph{principal ideal counting function} will be denoted by $ P_K(x) $.

The above formulas show that for a fixed $ K $, the number of ideals grows linearly in $ x $. To gain any useful information about $ I_K(x) $ and $ P_K(x,\mathfrak{c}) $ when $ K $ is not fixed, we need an explicit bound for the implied constant in the error term. Murty and van Order~\cite[Theorem 5]{MO} revised the argument of Weber to prove an estimate for $ P_K(x,\mathfrak{c}) $ with an effective error term. Debaene \cite[Corollary 2]{De19} proved an explicit estimate for the error term, which we formulate and apply in \Cref{sec:Apriori}.

An explicit version of \eqref{eq:IW} can be deduced from an explicit version of \eqref{eq:P}, see \cite{De19,MO}. An explicit version of \eqref{eq:IL} was proved in the thesis of Sunley \cite{Su71}, who showed that for $ x > 0 $ and $ d \geq 2 $, we have
\[
	|I_K(x)-\kappa_Kx| \leq \Lambda(d)|\Delta_K|^{\frac{1}{d+1}}\left(\log |\Delta_K|\right)^dx^{1-\frac{2}{d+1}},
\]
where $ \Lambda(d) $ is a certain constant depending only on the degree of the field. Recently, an upper bound for $ \Lambda(d) $ was significantly improved by Lee \cite{Le22}, who also reduced the power of $ \log |\Delta_K| $ from $ d $ to $ d-1 $.

Norms of principal ideals have been extensively studied in imaginary quadratic fields. Let $ K = \Q(\sqrt{-D}) $ be an imaginary quadratic field of discriminant $ -D $ and fix a basis $ (1,\omega_D) $ of the ring of integral elements $ \O_K $. The norm $ \Nm(\alpha) $ of $ \alpha = x_1+x_2\omega_D $ is a positive definite binary integral quadratic form of discriminant $ -D $. Besides $ P_K(x) $, one may also be interested in the number of integers $ 0 \leq n \leq x $ represented as the norm of some principal ideal in $ K $. More generally, let $ Q $ be a positive definite binary integral quadratic form of discriminant $ -D $, and let
\[
	N_Q(x) := \#\left\{0 \leq n \leq x : n = Q(m_1,m_2)\text{ for some }m_1, m_2 \in \Z\right\}.
\]

In the following, we use the standard analytic notation $ \ll $ and $ \asymp $. Let $ f: \R^2 \to \R $ and $ g: \R^2 \to \R $ be two functions and $ S \subset \R^2 $. We write $ f(a,x) = O(g(a,x)) $, where $ (a,x) \in S $ if there exist an absolute constant $ c > 0 $ such that for all $ (a, x) \in S $, we have $ |f(a,x)| \leq c g(a,x) $. The notation $ f(a,x) \ll g(a,x) $ has the same meaning as $ f(a,x) = O(g(a,x)) $. When $ c $ depends on some other quantities such as $ r $, then we write $ f(a,x) \ll_{r} g(a,x) $ etc. The notation $ f(a,x) \asymp g(a,x) $, where $ (a,x) \in S $ means that there exist constants $ c_1, c_2 > 0 $ such that for all $ (a,x) \in S $, we have $ c_1g(a,x) \leq f(a,x) \leq c_2g(a,x) $.

Bernays \cite{Ber12} showed that when $ Q $ is fixed, then there exists a constant $ c(D) > 0 $ (depending only on $ D $) such that
\[
	N_Q(x) = c(D)\cdot\frac{x}{\sqrt{\log x}}+O_D\left(\frac{x}{(\log x)^{\frac{1}{2}+\delta}}\right)
\]
for every $ \delta < \min(1/h,1/4) $, see \cite[p. 7]{BG}. Blomer and Granville \cite{BG} gave estimates for $ N_Q(x) $ depending on the relative size of $ D $ and $ x $. Let $ h $ be the class number of $ \Q(\sqrt{-D}) $, $ g $ the number of genera, and $ \ell_{-D} := L(1,\chi_{-D})\frac{\phi(D)}{D} $, where $ \chi_{-D} = \leg{-D}{\cdot} $ is the quadratic character of $ \Q(\sqrt{-D}) $. They introduce the quantity
\[
	\kappa := \frac{\log(h/g)}{(\log 2)\log(\ell_{-D}\log x)},
\]
and split the estimates for $ N_f(x) $ into three different ranges depending on the size of $ \kappa $:
\begin{align*}
	N_Q(x)& \asymp \frac{L(1,\chi_D)}{\tau(D)}\frac{x}{\sqrt{\ell_{-D}\log x}}&\text{for }0 \leq \kappa \leq 1/2,\\
	N_Q(x)& \asymp \frac{L(1,\chi_{-D})}{\tau(D)}\frac{(\ell_{-D}\log x)^{-1+\kappa(1-\log(2\kappa))}}{\left(1+\left(\kappa-\frac{1}{2}\right)(1-\kappa)\sqrt{\log\log x}\right)}x&\text{for }1/2 < \kappa < 1,\\
	N_Q(x)& \asymp \frac{x}{\sqrt{D}}&\text{for }1 \leq \kappa \ll \log D/\log\log D.
\end{align*}
These bounds are proved in \cite[Theorem 6]{BG}, except for the lower bound in the range $ \kappa \in \left[\frac{1}{2}-\varepsilon,\frac{1}{\log 2}+\varepsilon\right] $.

In this paper, our aim is to estimate $ P_K(x) $ when $ K $ belongs to the family of \emph{the simplest cubic fields} introduced by Shanks \cite{Sh74}. The field $ K_a $ is defined for $ a \in \Z_{\geq -1} $ to be the splitting field of the polynomial
\[
	f_a(X) = X^3-aX^2-(a+3)X-1.
\]
One important feature of these fields is that they are Galois and their Galois group is cyclic. Our approach requires working directly with the integral basis of $ K_a $. Thus, we always assume that $ K_a $ is \emph{monogenic}, i.e., the ring of integers is $ \O_{K_a} = \Z[\rho] $, where $ \rho $ is the largest root of $ f_a $. This assumption is relatively mild because $ K_a $ has this property for a positive density of $ a $. One also has an explicit description of the fundamental units in this situation, and hence precise bounds for the regulator. Many authors used these properties to study the fields $ K_a $, especially their class numbers \cite{Ba16, By00, Fo11, Ki03, Wa87}.

If we express an element $ \alpha \in \O_{K_a} $ in the basis $ (1,\rho,\rho^2) $ as $ \alpha = m+n\rho+o\rho^2 $, then it is possible to explicitly compute $ \Nm(\alpha) $, and we obtain a cubic form in the variables $ m, n, o $:
\begin{multline*}
	\Nm(m+n\rho+o\rho^2) = m^3+am^2n+(a^2+2a+6)m^2o-(a+3)mn^2\\-(a^2+3a+3)mno+(a^2+4a+9)mo^2+n^3+an^2o-(a+3)no^2+o^3.
\end{multline*}
Thus, we are essentially counting solutions $ (m,n,o) \in \Z^3 $ to $ \Nm(m+n\rho+o\rho^2) \leq x $.

To emphasize that the principal ideal counting function now depends on the two variables $ a $ and $ x $, we write $ P(a, x) $ instead of $ P_{K_a}(x) $. Similarly, we let $ \Reg_a $ and $ \Delta_a $ be the regulator and the discriminant of $ K_a $.

It is clear that if $ x $ is significantly larger than $ a $, the main term in \eqref{eq:P} overweights the error term. However, it is not immediately clear how large $ x $ needs to be for this to happen. We specialize \cite[Corollary 2]{De19} to the simplest cubic fields in \Cref{thm:Apr}, and show that for a certain explicit constant $ c $, we have
\[
	P(a, x) \asymp \frac{(\log a)^2 x}{a^2},\qquad x \geq ca^6(\log a)^4.
\]
This settles the problem for large norms and it is the small norms that remain of interest.

A closely related problem is to estimate the number of primitive principal integral ideals in $ K_a $ (we recall that an ideal $ I \subset \O_K $ is called \emph{primitive} if $ n \nmid I $ for every $ n \in \Z_{\geq 2} $). We let $ P_p(a, x) $ denote the associated counting function.

Lemmermeyer and Peth\"{o}~\cite[Theorem 1]{LP} show that if $ I = (\alpha) \subsetneq \O_{K_a} $ is a primitive principal ideal in $ K_a $ of minimal norm, then $ \Nm(I) = 2a+3 $. Hence, the only primitive principal ideal of norm $ 1 \leq \Nm(I) < 2a+3 $ is $ I = \O_{K_a} $.

A study of various properties of the simplest cubic fields was made by Kala and Tinkov\'{a}~\cite{KT}, see also \cite{GT22, Ti23a, Ti23b}. Their main tool are the so-called \emph{indecomposable elements}. In the present paper, we will not use indecomposables at all but let us mention that one motivation for studying them lies in the fact that they can be applied to the theory of universal quadratic forms over number fields \cite{BK1, BK2, Ki00}. The authors of \cite{KT} provide a characterization of indecomposable elements in $ K_a $ and apply them to several problems, including the problem to estimate the number of primitive principal ideals of norm $ \leq x $ for $ 2a+3 \leq x \leq a^2 $. More precisely, the bound $ x $ is assumed to be of the form $ x = a^{1+\delta} $ for $ 0 < \delta \leq 1 $ and it is proved that $ P(a, x) \asymp a^{2\delta/3} $. However, by examining the proof, one sees that it is not crucial to consider only $ x $ of this form. Thus~\cite[Theorem~1.3]{KT} amounts to
\[
	P_p(a, x) \asymp \left(\frac{x}{a}\right)^{2/3},\qquad a \leq x \leq a^2.
\]
This result is surprising because it shows that the number of ideals of small norm is quite large. For example, for $ x = a^2 $, we get $ \asymp a^{2/3} $ ideals, while based on the Class Number Formula, we would expect only $ \asymp (\log a)^2 $. Naturally, one asks how to obtain a similar estimate for $ x $ in a wider range and this is what inspired the present work.

On the other hand, if $ n \in \Z_{\geq 2} $, then the principal ideal $ I = (n) $ has norm $ \Nm(I) = n^3 $. The number of (non-primitive) principal ideals of norm $ \leq x $ is therefore at least $ \lfloor x^{1/3} \rfloor $. We will prove below that this is the true count if $ x \leq a^2 $, i.e.,
\[
	P(a, x) \asymp x^{1/3},\qquad 1 \leq x \leq a^2.
\]
In view of our discussion, what remains is to close the gap between $ a^2 $ and $ a^{6+\varepsilon} $. This is the content of our main theorem.

\begin{theorem}
\label{thm:P}
Let $ K_a = \Q(\rho) $ be a simplest cubic field, where $ \rho $ is the largest root of $ f_a $. Assume that $ \O_{K_a} = \Z[\rho] $. If $ a \geq 8 $ and $ x \geq 1 $, then
\[
	P(a, x) \asymp \frac{(\log a)^2x}{a^2}+\left(\frac{x}{a}\right)^{2/3}+x^{1/3}
\]
and
\[
	P_p(a, x) \asymp \frac{(\log a)^2x}{a^2}+\left(\frac{x}{a}\right)^{2/3}+1.
\]
\end{theorem}
\begin{proof}
This will be proved as \Cref{thm:P'}.
\end{proof}

The last term in the two estimates corresponds to $ C_K x $ from \eqref{eq:P}. We note that $ (\log a)^2x/a^2 $ is smaller than $ (x/a)^{2/3} $ for $ x < a^4/(\log a)^6 $. Only for $ x $ (roughly) larger than the discriminant $ \Delta_a \asymp a^4 $ is the order of magnitude of $ P(a, x) $ the one predicted by the Class Number Formula:
\[
	P(a, x) \asymp \frac{(\log a)^2x}{a^2},\qquad x \geq \frac{a^4}{(\log a)^6}.
\]
This is a better range than the one we got from the apriori estimate in \Cref{thm:Apr}.

Let us conclude with some open problems.
\begin{enumerate}
\item Estimates in other families: The method that we applied to the simplest cubic fields can in principle be used for every family where we know an integral basis and generators of the unit group, hence it should be possible to obtain an analogy of \Cref{thm:P} in these cases. Examples include
\begin{itemize}
	\item the family $ \Q(\sqrt{a^2+1}) $ of real quadratic fields,
	\item Ennola's cubic fields generated by a root of $ x^3+(a-1)x^2-ax-1 $, where $ a \geq 3 $.
\end{itemize}
\item Is it possible to estimate the number of ``small" principal ideals in a general cubic field? By ``small" we mean ideals of norm $ \Nm(I) \leq x $ for $ x \leq \Delta_K $. We have seen that there is a large number of such ideals in the simplest cubic fields and it would be interesting to find out to what extent this phenomenon occurs in general. 
\item Vague question: is it possible to predict the general shape of the formula for $ P(a,x) $ without going through the technical computations in the proof of \Cref{thm:P}?
\end{enumerate}

The rest of the paper is organized as follows. In \Cref{sec:Prelim}, we collect the preliminaries. In \Cref{sec:Apriori}, we prove an apriori estimate for $ P(a,x) $. The proof of \Cref{thm:P} is split into the remaining sections. In \Cref{sec:Setup}, we translate the problem of counting ideals into the problem of counting lattice points inside a certain region in $ \R^3 $. Some auxiliary lemmas are contained in \Cref{sec:Count}. The bulk of the work is done in \Cref{sec:UB}, where we obtain the upper bound for $ P(a,x) $. The lower bound is proved along with the main theorem in \Cref{sec:LB}.

\section*{Acknowledgments}
I would like to thank my PhD advisor V\'{i}t\v{e}zslav Kala for his guidance and support. I also thank Jeanine Van Order for kindly answering my questions about her paper \cite{MO}.

\section{Preliminaries}
\label{sec:Prelim}

Let $ K $ denote a number field of degree $ d = [K:\Q] $ and $ \O_K $ its ring of integers. We will always assume that $ K $ is totally real. The \emph{numerical norm} of an ideal $ I \subset \O_K $ is denoted $ \Nm(I) $. It is always a positive integer. In the case when $ I = (\alpha) $ is principal, we have $ \Nm(I) = |\Nm(\alpha)| $, where $ \Nm(\alpha) $ is the norm of $ \alpha $.

The number field $ K $ has $ d $ real embeddings
\[
	\sigma_i: K \hookrightarrow \R,\qquad 1 \leq i \leq d.
\]
An element $ \alpha \in K $ is called \emph{totally positive} if $ \sigma_i(\alpha) > 0 $ for every $ i $, $ 1 \leq i \leq d $. If the ring of integers contains units of all signatures (as is the case for the simplest cubic fields), we can multiply any element $ \alpha \in K $ by a suitable unit $ \varepsilon $ such that $ \varepsilon\alpha $ is totally positive. Consequently, any principal ideal is generated by a totally positive element $ \alpha $, and we can write simply $ \Nm(\alpha) $ for its norm.

The discriminant of $ f_a $ is $ \Delta_{f_a} = (a^2+3a+9)^2 $. If $ a^2+3a+9 $ is squarefree, then the discriminant of the field $ K_a $ satisfies $ \Delta_a = \Delta_{f_a} = (a^2+3a+9)^2 $ and $ \O_{K_a} = \Z[\rho] $. This happens for a set of $ a $ with positive density. In this paper, we always assume that $ \O_{K_a} = \Z[\rho] $.

The polynomial $ f_a $ has three real roots $ \rho > |\rho'| > |\rho''| $. Since $ K_a $ is the splitting field of $ f_a $, all the roots lie in $ K_a $. The three conjugates are units of norm $ 1 $ and we have the following estimates for their size.

\begin{lemma}[{\cite[p. 54]{LP}}]
\label{lem:Units}
If $ a \geq 7 $, then
\begin{align*}
	a+1& < \rho < a+1+\frac{2}{a},\\
	-1-\frac{1}{a}& < \rho' < -1-\frac{1}{2a},\\
	-\frac{1}{a+2}& < \rho'' < -\frac{1}{a+3},
\end{align*}
and if $ a \geq 8 $, then
\begin{align*}
	a^2& < \rho^2 < 2a^2,\\
	1& < (\rho')^2 < 2,\\
	\frac{1}{2a^2}& < (\rho'')^2 < \frac{1}{a^2}.
\end{align*}
\end{lemma}

The group of units $ \O_{K_a}^\times $ is generated by $ \rho $ and $ \rho' $~\cite{Go60, Sh74}. The fact that there are units of all signatures implies that every totally positive unit is a square~\cite[p.~11., Corollary 3]{Na04}. Consequently, the group of totally positive units $ \O_{K_a}^{\times,+} $ is generated by $ \rho^2 $ and $ (\rho')^{2} $.

We use $ A \sqcup B $ to denote disjoint union, i.e., the union of two sets $ A $ and $ B $ such that $ A \cap B = \emptyset $.

\section{Apriori estimate}
\label{sec:Apriori}

In this section, we apply the following estimate from \cite{De19} to the number of principal ideals in the simplest cubic fields.

\begin{lemma}[{\cite[Corollary 2]{De19}}]
\label{lem:Deb}
Let $ K $ be a number field of degree $ d $. Then, for all $ x \geq 1 $,
\[
	\left|I_K(x)-\kappa_Kx\right| \leq d^{10d^2}\left(\Reg_K h_K\right)^{1/d}\left(1+\log\Reg_K h_K\right)^{\frac{(d-1)^2}{d}}x^{1-\frac{1}{d}}.
\]
Moreover, let $ \mathfrak{c} $ be an element of the class group of $ K $. Then,
\[
	\left|P_K(x,\mathfrak{c})-C_Kx\right| \leq d^{10d^2}\left(\Reg_K\right)^{1/d}\left(1+\log\Reg_K\right)^{\frac{(d-1)^2}{d}}x^{1-\frac{1}{d}}.
\]
\end{lemma}

\begin{lemma}
\label{lem:Reg}
Let $ K_a = \Q(\rho) $ be a simplest cubic field, where $ a \in \Z $, $ a \geq 7 $, such that $ \O_{K_a} = \Z[\rho] $. The regulator of $ K_a $ satisfies
\[
	(\log a)^2 \leq \Reg_a \leq \frac{3}{2}(\log a)^2.
\]
\end{lemma}
\begin{proof}
We have
\[
	\Reg_a = \left|\det\begin{pmatrix}
		\log\rho&\log|\rho'|\\
		\log|\rho'|&\log|\rho''|
	\end{pmatrix}\right| = \left(\log|\rho'|\right)^2-(\log\rho)(\log|\rho''|).
\]
From the estimates on the sizes of units in \Cref{lem:Units}, we get
\[
	\log^2\left(1+\frac{1}{2a}\right)+\log(a+1)\log(a+2) \leq \Reg_a \leq \log^2\left(1+\frac{1}{a}\right)+\log(a+2)\log(a+3).
\]
The function
\[
	g(a) := \frac{\log^2\left(1+\frac{1}{a}\right)+\log(a+2)\log(a+3)}{(\log a)^2}
\]
is decreasing for $ a \in (1,+\infty) $. We assume $ a \geq 7 $, hence
\[
	g(a) \leq g(7) = \frac{\log^2\left(1+\frac{1}{7}\right)+\log(9)\log(10)}{(\log 7)^2} < \frac{3}{2}
\]
and
\[
	\log^2\left(1+\frac{1}{a}\right)+\log(a+2)\log(a+3) \leq \frac{3}{2}(\log a)^2.\qedhere
\]
\end{proof}

\begin{theorem}
\label{thm:Apr}
Let $ K_a = \Q(\rho) $ be a simplest cubic field, where $ a \in \Z $, $ a \geq 7 $. If $ \O_{K_a} = \Z[\rho] $, then
\[
	\left|P(a,x)-C_{K_a}x\right| \leq c_1(\log a)^{10/3}x^{2/3},
\]
where
\[
	c_1 := 3^{900}\left(\frac{3}{2}\right)^{5/3}.
\]
In particular, if $ x \geq c_1^3a^6(\log a)^4 $, then
\[
	P(a,x) \asymp \frac{(\log a)^2}{a^2}x.
\]
\end{theorem}
\begin{proof}
By \Cref{lem:Deb} for the class of principal ideals, we have
\[
	\left|P(a,x)-C_{K_a}x\right| \leq 3^{900}(\Reg_a)^{1/3}\left(1+\log\Reg_a\right)^{4/3}x^{2/3}.
\]
Crudely estimating $ 1+\log\Reg_a \leq \Reg_a $, we get
\[
	\left|P(a,x)-C_{K_a}x\right| \leq 3^{900}(\Reg_a)^{5/3}x^{2/3}.
\]
By \Cref{lem:Reg},
\[
	\left|P(a,x)-C_{K_a}x\right| \leq 3^{900}\left(\frac{3}{2}\right)^{5/3}(\log a)^{10/3}x^{2/3} = c_1(\log a)^{10/3}x^{2/3}.
\]

We have $ r_1 = 3 $, $ r_2 = 0 $, and $ w = 2 $ in the formula for $ C_{K_a} $, hence
\[
	C_{K_a} = \frac{4\Reg_a}{\sqrt{\Delta_a}} = \frac{4\Reg_a}{a^2+3a+9}.
\]
By \Cref{lem:Reg},
\[
	\frac{2(\log a)^2}{a^2} \leq C_{K_a} \leq \frac{6(\log a)^2}{a^2}.
\]
If $ x \geq c_1^3a^6(\log a)^4 $, then
\[
	c_1(\log a)^{10/3}x^{2/3} \leq \frac{(\log a)^2}{a^2}x.
\]
Thus, we get
\[
	P(a,x) \leq C_{K_a}x+c_1(\log a)^{10/3}x^{2/3} \leq \frac{7(\log a)^2}{a^2}x
\]
and
\[
	P(a,x) \geq C_{K_a}x-c_1(\log a)^{10/3}x^{2/3} \geq \frac{(\log a)^2}{a^2}x. \qedhere
\]
\end{proof}

\section{Setting up the counting problem}
\label{sec:Setup}

For the moment, let us work in a totally real field $ K $ of degree $ d $. In order to estimate $ P(a, x) $, we translate the problem of counting principal ideals $ I = (\alpha) $ with norm $ \Nm(I) \leq x $ into the problem of counting lattice points in a certain region of $ \R^d $. In this we closely follow the exposition in~\cite{KT}.

Two integral elements $ \alpha, \beta \in \O_K $ generate the same principal ideal if and only if they are associated, i.e., there exists a unit $ \varepsilon \in \O_K^\times $ such that $ \beta = \varepsilon\alpha $. Here we make the assumption that $ K $ contains units of all signatures. For $ \alpha \in \O_K $, there exists a unit $ \varepsilon \in \O_K^\times $ such that $ \varepsilon\alpha $ is totally positive. Thus, we want to count $ \alpha \in \O_K^+ $ satisfying $ \Nm(\alpha) \leq x $ up to multiplication by totally positive units.

Next, we move the problem into the Minkowski space. Let $ \sigma: K \hookrightarrow \R^d $ be the embedding
\[
	\sigma: \alpha \mapsto (\sigma_1(\alpha), \sigma_2(\alpha), \ldots, \sigma_d(\alpha)).
\]
The elements $ \alpha \in \O_K^+ $ are mapped into $ \R_{>0}^{d} $. The next step is to determine the fundamental domain for the action of multiplication by totally positive units $ \varepsilon \in \O_K^{\times, +} $ on $ \R_{>0}^{d} $. We recall the required terminology: A \emph{simplicial cone} in $ \R_{>0}^{d} $ is $ \calC = \R_{>0}\ell_1+\dots+\R_{>0}\ell_e $, where $ \ell_1, \ldots, \ell_e $ are linearly independent vectors in $ \R^d $. A \emph{polyhedric cone} is a finite disjoint union of simplicial cones. By Shintani's Unit Theorem~\cite[Chapter VII, Theorem 9.3]{Ne99}, it is possible to choose a polyhedric cone $ \P $ to be the fundamental domain. Let $ \varepsilon_1, \ldots, \varepsilon_{d-1} $ be a system of fundamental totally positive units, i.e, generators of $ \O_K^+ $. The polyhedric cone $ \P $ is the union of simplicial cones generated by some of the vectors
\[
	\sigma\left(\prod_{i \in I}\varepsilon_i\right)
\]
for subsets $ I \subset \{1, \ldots, d-1\} $.

Naturally, it is more convenient to work with elements expressed in an integral basis. We define a map $ \tau: K \hookrightarrow \R^d $ by
\[
	\tau: \sum_{i=1}^d x_i\omega_i \mapsto (x_1, \ldots, x_d),
\]
where $ (\omega_1, \ldots, \omega_d) $ is a fixed integral basis of $ K $. The maps $ \sigma $ and $ \tau $ are linear and invertible, hence $ \mathcal{Q} := \tau\sigma^{-1}\P $ is also a polyhedric cone. If $ \alpha_1, \alpha_2, \dots, \alpha_e \in K $, then we let
\[
	\calC(\alpha_1, \ldots, \alpha_e) := \R_{>0} \tau(\alpha_1)+\dots+\R_{>0}\tau(\alpha_e).
\]

In the case of cubic fields, i.e., for $ d = 3 $, we can take (as in \cite[p. 7551]{KT})
\[
	\mathcal{Q} = \calC(1, \varepsilon_1, \varepsilon_2)\sqcup \calC(1, \varepsilon_1, \varepsilon_1\varepsilon_2^{-1})\sqcup \calC(1, \varepsilon_1)\sqcup \calC(1, \varepsilon_2)\sqcup \calC(1, \varepsilon_1\varepsilon_2^{-1})\sqcup \calC(1)
\]
for a suitable pair of units $ (\varepsilon_1,\varepsilon_2) $. In the particular case of the simplest cubic fields, we can choose $ \varepsilon_1 = \rho^2 $ and $ \varepsilon_2 = (\rho'')^{-2} $, hence $ \varepsilon_1\varepsilon_2^{-1} = (\rho')^{-2} $, and
\[
	\calQ = \calC(1, \rho^2, (\rho'')^{-2})\sqcup \calC(1, \rho^2, (\rho')^{-2})\sqcup \calC(1, \rho^2)\sqcup \calC(1, (\rho'')^{-2})\sqcup \calC(1, (\rho')^{-2})\sqcup \calC(1).
\]
The units are expressed in the basis $ (1, \rho, \rho^2) $ as follows:
\begin{align*}
	\tau(\rho^2)& = (0, 0, 1),\\
	\tau((\rho'')^{-2})& = (1, 2, 1),\\
	\tau((\rho')^{-2})& = (-a-1, -(a^2+3a+3), a+2).
\end{align*}
The conclusion of our discussion is that
\[
	P(a, x) = \#\left\{\alpha\in \O_{K_a}^+ : \Nm(\alpha) \leq x, \tau(\alpha) \in \mathcal{Q}\right\}.
\]
What remains is to estimate the number of lattice points in each of the simplicial cones whose union is $ \calQ $.

If $ \calC \subset \R^3 $ is a simplicial cone, then we let
\[
	P(a, x, \calC) := \#\left\{\alpha\in\O_{K_a}^+ : \Nm(\alpha) \leq x, \tau(\alpha) \in \calC\right\}.
\]
\begin{lemma}
\label{lem:C1}
Let $ a \in \Z $, $ a \geq 8 $ and $ x \geq 1 $. We have
\[
	P(a, x, \calC(1)) \asymp x^{1/3}.
\]
If $ \calC \in \left\{\calC(1, \rho^2), \calC(1, (\rho'')^{-2}), \calC(1, (\rho')^{-2})\right\} $, then
\[
	P(a, x, \calC) \ll \left(\frac{x}{a^2}\right)^{2/3}.
\]
\end{lemma}
\begin{proof}
If $ \alpha \in \O_{K_a}^+ $, then we let $ \tau(\alpha) = (m, n, o) \in \Z^3 $. We begin by counting the lattice points in $ \calC(1) $. If $ \tau(\alpha) \in \calC(1) $, then there exists $ t \in \R_{>0} $ such that
\[
	(m, n, o) = t(1, 0, 0),
\]
that is, $ \alpha = m $. The number of integers $ m \in \Z_{\geq 1} $ such that $ \Nm(m) \leq x $ is $ \asymp x^{1/3} $.

Next, we let $ \calC = \calC(1, \rho^2) $ and estimate $ P(a, x, \calC) $ from above. If $ \tau(\alpha) \in \calC(1, \rho^2) $, then there exist $ t_1, t_2 \in \R_{>0} $ such that
\[
	(m, n, o) = t_1(1, 0, 0)+t_2(0, 0, 1).
\]
Thus, $ \alpha = m+o\rho^2 $, where $ m, o \in \Z_{\geq 1} $. In view of the bounds for the sizes of units in \Cref{lem:Units}, we have
\begin{align*}
	\alpha& > m+oa^2,\\
	\alpha'& > m+o,\\
	\alpha''& > m+\frac{o}{2a^2}.
\end{align*}
If $ \Nm(\alpha) \leq x $, then
\[
	(m+oa^2)(m+o)\left(m+\frac{o}{2a^2}\right) \leq x.
\]
Since all the terms on the left side are positive, we get in particular $ a^2m^2o \leq x $ and $ a^2mo^2 \leq x $. Therefore
\[
	P(a, x, \calC) \leq \#M(a, x)
\]
where
\[
	M(a, x) := \left\{(m, o) \in \Z_{\geq 1}^2 : a^2m^2o \leq x, a^2mo^2 \leq x\right\}.
\]
Let $ (m, o) \in M(a, x) $. If $ m \leq o $, then $ a^2m^3 \leq x $, hence $ m \leq (x/a^2)^{1/3} $. Since $ o \leq (x/a^2)^{1/2}m^{-1/2} $, we get
\[
	 \#\left\{(m, o) \in M(a, x), m \leq o\right\} \leq \sum_{m \leq (a^2/x)^{1/3}} \left(\frac{x}{a^2}\right)^{1/2}m^{-1/2} \ll \left(\frac{x}{a^2}\right)^{2/3}.
\]
If $ m \geq o $, the estimate is the same by symmetry, and hence
\[
	P(a, x, \calC) \ll \left(\frac{x}{a^2}\right)^{2/3}.
\]
Next, we show that counting the lattice points in $ \calC(1, (\rho'')^{-2}) $ and $ \calC(1, (\rho')^{-2}) $ reduces to the preceding case. If $ \alpha \in \O_{K_a}^+ $ satisfies $ \tau(\alpha) \in \calC(1, (\rho'')^{-2}) $, then there exist $ t_1, t_2 \in \R_{>0} $ such that $ \alpha = t_1+t_2(\rho'')^{-2} $, and hence
\[
	\alpha' = t_1+t_2\rho^{-2} \Longrightarrow \rho^2\alpha' = t_1\rho^2+t_2 \Longrightarrow \tau(\rho^2\alpha') \in \calC(1, \rho^2).
\]
Conversely, $ \tau(\rho^2\alpha') \in \calC(1, \rho^2) $ implies $ \tau(\alpha) \in \calC(1, (\rho'')^{-2}) $. By similar reasoning, $ \tau(\alpha) \in \calC(1, (\rho')^{-2}) $ if and only if $ \tau(\rho^2\alpha'') \in \calC(1, \rho^2) $. This proves
\[
	P(a, x, \calC(1, \rho^2)) = P(a, x, \calC(1, (\rho'')^{-2})) = P(a, x, \calC(1, (\rho')^{-2})).\qedhere
\]
\end{proof}

\begin{lemma}
\label{lem:C2}
Let $ a \in \Z $, $ a \geq 8 $, and $ r > 0 $. If $ \calC = \calC(1, \rho^2, (\rho'')^{-2}) $ and $ 1 \leq x \leq a^r $, then
\[
	P(a, x, \calC) \ll_r \frac{(\log a)^2 x}{a^2}.
\]
The constant implied in the estimate depends only on $ r $.
\end{lemma}
\begin{proof}
If $ \alpha \in \O_{K_a}^+ $, then we let $ \tau(\alpha) = (m, n, o) \in \Z^3 $. If $ \alpha \in \calC $, then there exist $ t_1, t_2, t_3 \in \R_{>0} $ such that
\[
	(m, n, o) = \tau(\alpha) = t_1\tau(1)+t_2\tau(\rho^2)+t_3\tau((\rho'')^{-2}) = t_1(1,0,0)+t_2(0,0,1)+t_3(1,2,1).
\]
Thus, $ m = t_1+t_3 $, $ n = 2t_3 $, and $ o = t_2+t_3 $, which implies $ t_3 = \frac{n}{2} $, $ t_1 = m-\frac{n}{2} $, and $ t_2 = o-\frac{n}{2} $. Since $ t_3>0 $, we get $ n \in \Z_{\geq 1} $. Write $ m = \lfloor\frac{n}{2}\rfloor+m_1 $ and $ o = \lfloor\frac{n}{2}\rfloor+o_1 $, where $ m_1, o_1 \in \Z $. Now
\begin{align*}
	t_1& = m-\frac{n}{2} = m_1-\left(\frac{n}{2}-\left\lfloor\frac{n}{2}\right\rfloor\right) \geq m_1-\frac{1}{2},\\
	t_2& = o-\frac{n}{2} = o_1-\left(\frac{n}{2}-\left\lfloor\frac{n}{2}\right\rfloor\right) \geq o_1-\frac{1}{2}.
\end{align*}
Since $ t_1>0 $ and $ t_2>0 $, we have $ m_1, o_1 \in \Z_{\geq 1} $, and $ t_1 \geq \frac{m_1}{2} $, $ t_2 \geq \frac{o_1}{2} $.

Using the bounds for the sizes of the units in \Cref{lem:Units}, we get
\begin{align*}
	\alpha& = t_1+t_2\rho^2+t_3(\rho'')^{-2} > t_1+t_2a^2+t_3a^2 \geq \frac{m_1}{2}+\frac{o_1a^2}{2}+\frac{na^2}{2},\\
	\alpha'& = t_1+t_2(\rho')^2+t_3\rho^{-2} > t_1+t_2+\frac{t_3}{2a^2} \geq \frac{m_1}{2}+\frac{o_1}{2}+\frac{n}{4a^2},\\
	\alpha''& = t_1+t_2(\rho'')^2+t_3(\rho')^{-2} > t_1+\frac{t_2}{2a^2}+\frac{t_3}{2} \geq \frac{m_1}{2}+\frac{o_1}{4a^2}+\frac{n}{4}.
\end{align*}

If $ \Nm(\alpha) \leq x $, then
\[
	\left(\frac{m_1}{2}+\frac{o_1a^2}{2}+\frac{na^2}{2}\right)\left(\frac{m_1}{2}+\frac{o_1}{2}+\frac{n}{4a^2}\right)\left(\frac{m_1}{2}+\frac{o_1}{4a^2}+\frac{n}{4}\right) \leq x.
\]
In particular, $ m_1^3 \leq 8x $, $ o_1^3 \leq 16x $, and $ a^2o_1m_1n \leq 16x $. Thus,
\[
	P(a, x, \calC) \leq M(a, x),
\]
where
\[
	M(a, x) := \#\left\{(m_1, n, o_1) \in \Z_{\geq 1}^3 : m_1^3 \leq 8x, o_1^3 \leq 16x, a^2o_1m_1n \leq 16x\right\}.
\]

We estimate
\[
	M(a, x) \leq \sum_{1 \leq m_1 \leq (8x)^{1/3}}\sum_{1 \leq o_1 \leq (16x)^{1/3}}\frac{16x}{a^2m_1o_1} \ll \frac{x}{a^2}\log(8x)\log(16x) \ll_r \frac{(\log a)^2 x}{a^2}.\qedhere
\]
\end{proof}

\section{Counting lattice points in a three-dimensional cone}
\label{sec:Count}

In \Cref{sec:Setup}, we gave upper bounds for $ P(a,x,\calC) $ when $ \calC $ is one of the cones $ \calC(1) $, $ \calC(1,\rho^2) $, $ \calC(1,(\rho'')^{-2}) $, $ \calC(1,(\rho')^{-2}) $, and $ \calC(1,\rho^2, (\rho'')^{-2}) $. Obtaining an upper bound for $ \calC(1, \rho^2, (\rho')^{-2}) $ is significantly harder. In this section, we lay the groundwork for a calculation which will be completed in \Cref{sec:UB}.

From now on, let $ \calC := \calC(1, \rho^2, (\rho')^{-2}) $. Let $ \alpha \in \O_{K_a}^+ $ and $ \tau(\alpha) = (m, n, o) \in \Z^3 $. If $ \tau(\alpha) \in \calC $, then
\[
	(m, n, o) = t_1(1,0,0)+t_2(0,0,1)+t_3(-a-1,-(a^2+3a+3),a+2).
\]
Thus, $ m = t_1-t_3(a+1) $, $ n = -t_3(a^2+3a+3) $, and $ o = t_2+t_3(a+2) $. We set $ n = -w $, where $ w \in \Z_{\geq 1} $ to get
\begin{align*}
	t_3& = \frac{w}{a^2+3a+3},\\
	t_1& = m+\frac{a+1}{a^2+3a+3}w,\\
	t_2& = o-\frac{a+2}{a^2+3a+3}w.
\end{align*}

We introduce the following pieces of notation. For $ a \in \Z_{\geq 1} $ and $ w \in \Z_{\geq 1} $, let
\begin{align*}
	m_0(a, w)& := \left\lfloor -\frac{a+1}{a^2+3a+3}w \right\rfloor,\\
	o_0(a, w)& := \left\lfloor \frac{a+2}{a^2+3a+3}w \right\rfloor.
\end{align*}

Since $ a^2+3a+3 = (a+1)(a+2)+1 $, we can express every $ w \in \Z_{\geq 1} $ uniquely as
\[
	w = (a^2+3a+3)w_0+(a+1)w_1+w_2,
\]
where $ w_0, w_1, w_2 \in \Z_{\geq 0} $, $ w_1 \leq a+2 $, $ w_2 \leq a $, and $ w_2 = 0 $ if $ w_1 = a+2 $. Let
\[
	W_1(a) := \left\{(w_0, w_1, w_2) \in \Z_{\geq 0}^3 : w_1 \leq a+2, w_2 \leq a, w_2 = 0\text{ if }w_1 = a+2\right\}
\]
and define the map
\[
	\phi_a: \Z_{\geq 1} \to W_1(a), w \mapsto (w_0, w_1, w_2),
\]
where $ w = (a^2+3a+3)w_0+(a+1)w_1+w_2 $.

Similarly, every $ w \in \Z_{\geq 1} $ can be uniquely expressed as
\[
	w =(a^2+3a+3)w_0+(a+2)w_1+w_2,
\]
where $ w_0, w_1, w_2 \in \Z_{\geq 0} $, $ w_1 \leq a+1 $, $ w_2 \leq a+1 $, and $ w_2 = 0 $ if $ w_1 = a+1 $. Let
\[
	W_2(a) := \left\{(w_0, w_1, w_2) \in \Z_{\geq 0}^3 : w_1 \leq a+1, w_2 \leq a+1, w_2 = 0\text{ if }w_1 = a+1\right\}
\]
and
\[
	\psi_a: \Z_{\geq 1} \to W_2(a), w \mapsto (w_0, w_1, w_2),
\]
where $ w = (a^2+3a+3)w_0+(a+2)w_1+w_2 $.

If $ S \subset \Z^5 $, then we let
\begin{align*}
	\Phi(a, S) := \left\{(m, w, o) \in \Z^3 : m = m_0(a, w)+m_1, o = o_0(a, w)+o_1, w \in \Z_{\geq 1}, (m_1, o_1, \phi_a(w)) \in S\right\},\\
	\Psi(a, S) := \left\{(m, w, o) \in \Z^3 : m = m_0(a, w)+m_1, o = o_0(a, w)+o_1, w \in \Z_{\geq 1}, (m_1, o_1, \psi_a(w)) \in S\right\},
\end{align*}
and
\begin{align*}
	Q(a, S) := \left\{\alpha \in \O_{K_a}^+ : \tau(\alpha) = (m, n, o) \in \calC, (m, -n, o) \in \Phi(a,S)\right\},\\
	R(a, S) := \left\{\alpha \in \O_{K_a}^+ : \tau(\alpha) = (m, n, o) \in \calC, (m, -n, o) \in \Psi(a,S)\right\}.
\end{align*}

We also let
\begin{align*}
	S_1(a)& := \left\{(m_1, o_1, w_0, w_1, w_2) \in \Z^2\times W_1(a) : m_1 \geq 2, o_1 \geq 2\right\},\\
	S_2(a)& := \left\{(m_1, o_1, w_0, w_1, w_2) \in \Z^2\times W_1(a) : m_1 \geq 2, o_1 = 1\right\},\\
	S_3(a)& := \left\{(m_1, o_1, w_0, w_1, w_2) \in \Z^2\times W_2(a) : m_1 = 1, o_1 \geq 2\right\},\\
	S_4(a)& := \left\{(m_1, o_1, w_0, w_1, w_2) \in \Z^2\times W_1(a) : m_1 = 1, o_1 = 1\right\}.
\end{align*}

\begin{lemma}
\label{lem:Estimates}
Let $ a \in \Z $, $ a \geq 8 $, and $ x > 0 $. Let $ t_1, t_2, t_3 \in \R_{>0} $ and $ \alpha = t_1+t_2\rho^2+t_3(\rho')^{-2} $. If $ \Nm(\alpha) \leq x $, then
\[
	t_1^3 < x,\ t_2^3 < x,\ a^2t_1t_2^2 < x,\ a^2t_1^2t_3 < x,\ a^2t_2t_3^2 < 2x,\ a^4t_1t_2t_3 < x.
\]
\end{lemma}
\begin{proof}
By the estimates for the sizes of units in \Cref{lem:Units},
\begin{align*}
	\alpha& = t_1+t_2\rho^2+t_3(\rho')^{-2} > t_1+t_2a^2+\frac{t_3}{2},\\
	\alpha'& = t_1+t_2(\rho')^2+t_3(\rho'')^{-2} > t_1+t_2+t_3a^2,\\
	\alpha''& = t_1+t_2(\rho'')^2+t_3\rho^{-2} > t_1+\frac{t_2}{2a^2}+\frac{t_3}{2a^2}.
\end{align*}
Assume $ \Nm(\alpha) = \alpha\alpha'\alpha'' \leq x $. First of all, $ t_1^3 < \alpha\alpha'\alpha'' \leq x $ and
\[
	t_2^3 = t_2\rho^2\cdot t_2(\rho')^2\cdot t_2(\rho'')^2 < \alpha\alpha'\alpha'' \leq x.
\]
Secondly,
\[
	\left(t_1+a^2t_2+\frac{t_3}{2}\right)\left(t_1+t_2+a^2t_3\right)\left(t_1+\frac{t_2}{2a^2}+\frac{t_3}{2a^2}\right) < x.
\]
Since all the terms are positive, we get all the other stated inequalities.
\end{proof}

\begin{lemma}
\label{lem:W1}
Let $ a \in \Z_{\geq 1} $, $ w \in \Z_{\geq 1} $, and $ (w_0, w_1, w_2) = \phi_a(w) $. We have
\begin{align*}
	\left\lfloor-\frac{a+1}{a^2+3a+3}w\right\rfloor = -(a+1)w_0-w_1+\begin{cases}
		1,&\text{if }w_1 = a+2,\\
		0,&\text{if }w_1 \leq a+1\text{ and }w_1 \geq w_2,\\
		-1,&\text{if }w_1 \leq a+1\text{ and }w_1 < w_2.
	\end{cases}
\end{align*}
and
\begin{align*}
	-\frac{a+1}{a^2+3a+3}w-\left\lfloor-\frac{a+1}{a^2+3a+3}w\right\rfloor = \frac{(a+2)w_1-(a+1)w_2}{a^2+3a+3}+\begin{cases}
		-1,&\text{if }w_1 = a+2,\\
		0,&\text{if }w_1 \leq a+1\text{ and }w_1 \geq w_2,\\
		1,&\text{if }w_1 \leq a+1\text{ and }w_1 < w_2.
	\end{cases}
\end{align*}
\end{lemma}
\begin{proof}
We have
\begin{align*}
	-\frac{a+1}{a^2+3a+3}w& = -(a+1)w_0-\frac{(a+1)^2}{a^2+3a+3}w_1-\frac{a+1}{a^2+3a+3}w_2\\
	& = -(a+1)w_0-w_1+\left(\frac{a+2}{a^2+3a+3}w_1-\frac{a+1}{a^2+3a+3}w_2\right).
\end{align*}

If $ w_1 = a+2 $, then $ w_2 = 0 $, and the expression in the bracket equals
\[
	\frac{a+2}{a^2+3a+3}w_1 = \frac{(a+2)^2}{a^2+3a+3} \in (1, 2),
\]
hence
\[
	\left\lfloor-\frac{a+1}{a^2+3a+3}w\right\rfloor = -(a+1)w_0-w_1+1
\]
and
\[
	-\frac{a+1}{a^2+3a+3}w-\left\lfloor-\frac{a+1}{a^2+3a+3}w\right\rfloor = \frac{(a+2)^2}{a^2+3a+3}-1 = \frac{a+1}{a^2+3a+3}.
\]

If $ w_1 \leq a+1 $, then
\[
	-1 < -\frac{(a+1)a}{a^2+3a+3} \leq \frac{a+2}{a^2+3a+3}w_1-\frac{a+1}{a^2+3a+3}w_2 \leq \frac{a+2}{a^2+3a+3}w_1 \leq \frac{(a+2)(a+1)}{a^2+3a+3} < 1.
\]

Moreover,
\[
	\frac{a+2}{a^2+3a+3}w_1-\frac{a+1}{a^2+3a+3}w_2 \geq 0 \Longleftrightarrow \frac{a+2}{a+1}w_1 \geq w_2 \Longleftrightarrow \left\lfloor\frac{a+2}{a+1}w_1\right\rfloor \geq w_2.
\]

We note that
\[
	\left\lfloor\frac{a+2}{a+1}w_1\right\rfloor = \left\lfloor w_1+\frac{w_1}{a+1}\right\rfloor = \begin{cases}
		w_1+1,&\text{if }w_1 = a+1,\\
		w_1,&\text{if }w_1 < a+1.
	\end{cases}
\]
Since $ w_2 \leq a $, we get
\[
	\frac{a+2}{a^2+3a+3}w_1-\frac{a+1}{a^2+3a+3}w_2 \geq 0 \Longleftrightarrow w_1 \geq w_2.
\]
Thus,
\[
	\left\lfloor-\frac{a+1}{a^2+3a+3}w\right\rfloor = -(a+1)w_0-w_1+\begin{cases}
		0,&\text{if }w_1 \geq w_2,\\
		-1,&\text{if }w_1 < w_2
	\end{cases}
\]
and
\[
	-\frac{a+1}{a^2+3a+3}w-\left\lfloor-\frac{a+1}{a^2+3a+3}w\right\rfloor = \frac{(a+2)w_1-(a+1)w_2}{a^2+3a+3}+\begin{cases}
		0,&\text{if }w_1 \geq w_2,\\
		1,&\text{if }w_1 < w_2.
	\end{cases}\qedhere
\]
\end{proof}

\begin{lemma}
\label{lem:W2}
Let $ a \in \Z_{\geq 1} $, $ w \in \Z_{\geq 1} $, and $ (w_0, w_1, w_2) = \phi_a(w) $. We have
\[
	\left\lfloor\frac{a+2}{a^2+3a+3}w\right\rfloor = (a+2)w_0+w_1+\begin{cases}
		0,&\text{if }w_1 = w_2 = 0,\\
		0,&\text{if }w_2 \geq 1,\\
		-1,&\text{if }w_1 \geq 1, w_2 = 0\\
	\end{cases}
\]
and
\[
	\frac{a+2}{a^2+3a+3}w-\left\lfloor\frac{a+2}{a^2+3a+3}w\right\rfloor = \begin{cases}
		0,&\text{if }w_1 = w_2 = 0,\\
		\frac{(a+2)w_2-w_1}{a^2+3a+3},&\text{if }w_2 \geq 1,\\
		1-\frac{w_1}{a^2+3a+3},&\text{if }w_1 \geq 1, w_2 = 0.
	\end{cases}
\]
\end{lemma}
\begin{proof}
We have
\begin{align*}
	\frac{a+2}{a^2+3a+3}w& = (a+2)w_0+\frac{(a+2)(a+1)}{a^2+3a+3}w_1+\frac{a+2}{a^2+3a+3}w_2\\
	& = (a+2)w_0+w_1+\left(\frac{a+2}{a^2+3a+3}w_2-\frac{1}{a^2+3a+3}w_1\right).
\end{align*}
If $ w_2 = 0 $, then the expression in the bracket satisfies
\[
	-1 < -\frac{a+2}{a^2+3a+3} \leq -\frac{w_1}{a^2+3a+3} \leq 0,
\]
hence
\[
	\left\lfloor\frac{a+2}{a^2+3a+3}w\right\rfloor = (a+2)w_0+\begin{cases}
	0,&\text{if }w_1 = 0,\\
	w_1-1,&\text{if }w_1 \geq 1
	\end{cases}
\]
and
\[
	\frac{a+2}{a^2+3a+3}w-\left\lfloor\frac{a+2}{a^2+3a+3}w\right\rfloor = \begin{cases}
		0,&\text{if }w_1 = 0,\\
		1-\frac{w_1}{a^2+3a+3},&\text{if }w_1 \geq 1.
	\end{cases}
\]
If $ w_2 \geq 1 $, then
\[
	0 < \frac{a+2}{a^2+3a+3}-\frac{a+1}{a^2+3a+3} \leq \frac{a+2}{a^2+3a+3}w_2-\frac{1}{a^2+3a+3}w_1 \leq \frac{(a+2)a}{a^2+3a+3} < 1,
\]
where we used that $ w_2 = 0 $ if $ w_1 = a+2 $. Hence,
\[
	\left\lfloor\frac{a+2}{a^2+3a+3}w\right\rfloor = (a+2)w_0+w_1
\]
and
\[
	\frac{a+2}{a^2+3a+3}w-\left\lfloor\frac{a+2}{a^2+3a+3}w\right\rfloor = \frac{(a+2)w_2-w_1}{a^2+3a+3}.\qedhere
\]
\end{proof}

\begin{lemma}
\label{lem:W3}
Let $ a \in \Z_{\geq 1} $, $ w \in \Z_{\geq 1} $, and $ (w_0, w_1, w_2) = \psi_a(w) $. We have
\[
	\left\lfloor-\frac{a+1}{a^2+3a+3}w\right\rfloor = -(a+1)w_0-w_1+\begin{cases}
		0,&\text{if }w_2 = 0,\\
		-1,&\text{if }w_2 \geq 1
	\end{cases}
\]
and
\[
	-\frac{a+1}{a^2+3a+3}w-\left\lfloor-\frac{a+1}{a^2+3a+3}w\right\rfloor = \begin{cases}
		\frac{w_1}{a^2+3a+3},&w_2 = 0,\\
		1+\frac{w_1-(a+1)w_2}{a^2+3a+3},&w_2 \geq 1.
	\end{cases}
\]
\end{lemma}
\begin{proof}
We have
\begin{align*}
	-\frac{a+1}{a^2+3a+3}w& = -(a+1)w_0-\frac{(a+1)(a+2)}{a^2+3a+3}w_1-\frac{a+1}{a^2+3a+3}w_2\\
	& = -(a+1)w_0-w_1+\left(\frac{w_1}{a^2+3a+3}-\frac{a+1}{a^2+3a+3}w_2\right).
\end{align*}

If $ w_2 = 0 $, then the term in the bracket satisfies
\[
	0 \leq \frac{w_1}{a^2+3a+3} \leq \frac{a+1}{a^2+3a+3} < 1,
\]
hence
\[
	\left\lfloor-\frac{a+1}{a^2+3a+3}w\right\rfloor = -(a+1)w_0-w_1
\]
and
\[
	-\frac{a+1}{a^2+3a+3}w-\left\lfloor-\frac{a+1}{a^2+3a+3}w\right\rfloor = \frac{w_1}{a^2+3a+3}.
\]

If $ w_2 \geq 1 $, then
\[
	-1 < -\frac{(a+1)^2}{a^2+3a+3} \leq \frac{w_1}{a^2+3a+3}-\frac{a+1}{a^2+3a+3}w_2 \leq \frac{a}{a^2+3a+3}-\frac{a+1}{a^2+3a+3}<0,
\]
where we used that $ w_2 = 0 $ if $ w_1 = a+1 $. Thus,
\[
	\left\lfloor-\frac{a+1}{a^2+3a+3}w\right\rfloor = -(a+1)w_0-w_1-1
\]
and
\[
	-\frac{a+1}{a^2+3a+3}w-\left\lfloor-\frac{a+1}{a^2+3a+3}w\right\rfloor = 1+\frac{w_1-(a+1)w_2}{a^2+3a+3}.\qedhere
\]
\end{proof}

\begin{lemma}
\label{lem:t1t2}
Let $ a \in \Z_{\geq 1} $. Let $ w \in \Z_{\geq 1} $, $ m = m_0(a, w)+m_1 $, and $ o = o_0(a, w)+o_1 $. If $ \alpha = t_1+t_2\rho^2+t_3(\rho')^{-2} $ and $ \tau(\alpha) = (m, -w, o) $, then $ m_1-1 \leq t_1 \leq m_1 $ and $ o_1-1 \leq t_2 \leq o_1 $.
\end{lemma}
\begin{proof}
We have
\[
	t_1 = m+\frac{a+1}{a^2+3a+3}w = m_0(a, w)+m_1+\frac{a+1}{a^2+3a+3}w = \left\lfloor -\frac{a+1}{a^2+3a+3}w \right\rfloor+m_1+\frac{a+1}{a^2+3a+3}w,
\]
hence $ m_1-1 \leq t_1 \leq m_1 $, and
\[
	t_2 = o-\frac{a+2}{a^2+3a+3}w = o_0(a,w)+o_1-\frac{a+2}{a^2+3a+3}w = \left\lfloor \frac{a+2}{a^2+3a+3}w \right\rfloor+o_1-\frac{a+2}{a^2+3a+3}w,
\]
hence $ o_1-1 \leq t_2 \leq o_1 $.
\end{proof}

\section{Upper bound}
\label{sec:UB}

\subsection{Estimate for \texorpdfstring{$ Q(a, S_1(a)) $}{Q(a,S1(a))}}
\label{sec:S1}

We recall that
\[
	S_1(a) := \left\{(m_1, o_1, w_0, w_1, w_2) \in \Z^2 \times W_1(a) : m_1 \geq 2, o_1 \geq 2\right\}.
\]
We write $ S_1(a) = S_{1,1}(a)\sqcup S_{1,2}(a)\sqcup S_{1,3}(a) $, where
\begin{align*}
	S_{1,1}(a)& := \left\{(m_1, o_1, w_0, w_1, w_2) \in S_1(a) : w_0 \geq 1\right\},\\
	S_{1,2}(a)& :=  \left\{(m_1, o_1, w_0, w_1, w_2) \in S_1(a) : w_0 = 0, w_1 \geq 1\right\},\\
	 S_{1,3}(a)& := \left\{(m_1, o_1, w_0, w_1, w_2) \in S_1(a) : w_0 = 0, w_1 = 0, w_2 \geq 1\right\}.
\end{align*}

\begin{lemma}
\label{lem:S11}
Let $ a \in \Z $, $ a \geq 8 $ and $ r > 0 $. If $ 1 \leq x \leq a^r $, then
\[
	\#\left\{\alpha \in Q(a, S_{1,1}(a)) : \Nm(\alpha) \leq x\right\} \ll_r \frac{x}{a^2}(\log a)^2.
\]
The constant implied in the estimate depends only on $ r $.
\end{lemma}
\begin{proof}
Let $ \alpha \in Q(a, S_{1,1}(a)) $, let $ t_1, t_2, t_3 \in \R_{>0} $ be such that $ \alpha = t_1+t_2\rho^2+t_3(\rho')^{-2} $, and let $ \tau(\alpha) = (m, -w, o) $, where $ (m, w, o) \in \Phi(a, S_{1,1}(a)) $. Let $ m = m_0(a, w)+m_1 $, $ o = o_0(a, w)+o_1 $ and $ m_2 = m_1-1 $, $ o_2 = o_1-1 $. Since $ m_1 \geq 2 $ and $ o_1 \geq 2 $, we have $ m_2 \geq 1 $ and $ o_2 \geq 1 $. By \Cref{lem:t1t2}, $ t_1 \geq m_1-1 = m_2 $ and $ t_2 \geq o_1-1 = o_2 $. We also have
\[
	t_3 = \frac{w}{a^2+3a+3} = w_0+\frac{(a+1)w_1}{a^2+3a+3}+\frac{w_2}{a^2+3a+3} \geq w_0.
\]

If $ \Nm(\alpha) \leq x $, then by \Cref{lem:Estimates}, $ t_1^3 < x $, $ t_2^3 < x $, and $ a^4t_1t_2t_3 < x $, hence $ m_2^3 < x $, $ o_2^3 < x $, and $ a^4m_2o_2w_0 < x $. Let
\[
	M(a, x) := \left\{(m_2, o_2, w) \in \Z_{\geq 1}^3 : m_2^3 < x, o_2^3 < x, a^4m_2o_2w_0 < x\right\}.
\]
We have
\[
	\#\left\{\alpha \in Q(a, S_{1,1}(a)) : \Nm(\alpha) \leq x\right\} \leq \#M(a, x).
\]

The number of triples in $ M(a, x) $ can be estimated as
\begin{align*}
	\#M(a, x)& \leq \sum_{1 \leq m_2 \leq x^{1/3}}\sum_{1 \leq o_2 \leq x^{1/3}}\sum_{0 \leq w_1 \leq a+2}\sum_{0 \leq w_2 \leq a}\frac{x}{a^4m_2o_2} \ll \frac{x}{a^2}\sum_{1 \leq m_2 \leq x^{1/3}}\sum_{1 \leq o_2 \leq x^{1/3}} \frac{1}{m_2o_2}\\
	& \ll \frac{x}{a^2}(\log 2x)^2 \ll_r \frac{x}{a^2}(\log a)^2.\qedhere
\end{align*}
\end{proof}

\begin{lemma}
\label{lem:S12}
Let $ a \in \Z $, $ a \geq 8 $ and $ r > 0 $. If $ 1 \leq x \leq a^r $, then
\[
	\#\left\{\alpha \in Q(a, S_{1,2}(a)) : \Nm(\alpha) \leq x\right\} \ll_r \frac{x}{a^2}(\log a)^2.
\]
The constant implied in the estimate depends only on $ r $.
\end{lemma}
\begin{proof}
Let $ \alpha \in Q(a, S_{1,2}(a)) $, let $ t_1, t_2, t_3 \in \R_{>0} $ be such that $ \alpha = t_1+t_2\rho^2+t_3(\rho')^{-2} $, and let $ \tau(\alpha) = (m, -w, o) $, where $ (m, w, o) \in \Phi(a, S_{1,2}(a)) $. Let $ m = m_0(a, w)+m_1 $, $ o = o_0(a, w)+o_1 $ and $ m_2 = m_1-1 $, $ o_2 = o_1-1 $. Since $ m_1 \geq 2 $ and $ o_1 \geq 2 $, we have $ m_2 \geq 1 $ and $ o_2 \geq 1 $. By \Cref{lem:t1t2}, $ t_1 \geq m_1-1 = m_2 $ and $ t_2 \geq o_1-1 = o_2 $. We also have
\[
	t_3 = \frac{w}{a^2+3a+3} = w_0+\frac{(a+1)w_1}{a^2+3a+3}+\frac{w_2}{a^2+3a+3} \geq \frac{(a+1)w_1}{a^2+3a+3} \geq \frac{w_1}{2a}.
\]

If $ \Nm(\alpha) \leq x $, then by \Cref{lem:Estimates}, $ t_1^3 < x $, $ t_2^3 < x $, and $ a^4t_1t_2t_3 < x $, hence $ m_2^3 < x $, $ o_2^3 < x $, and $ a^3m_2o_2w_1 < 2x $. Let
\[
	M(a, x) := \left\{(m_2, o_2, w) \in \Z_{\geq 1}^3 : m_2^3 < x, o_2^3 < x, w_0 = 0, a^3m_2o_2w_1 < 2x\right\}.
\]
We have
\[
	\#\left\{\alpha \in Q(a, S_{1,2}(a)) : \Nm(\alpha) \leq x\right\} \leq \#M(a, x).
\]

The number of triples in $ M(a, x) $ can be estimated as
\begin{align*}
	\#M(a, x)& \leq \sum_{1 \leq m_2 \leq x^{1/3}}\sum_{1 \leq o_2 \leq x^{1/3}}\sum_{0 \leq w_2 \leq a}\frac{2x}{a^3m_2o_2} \ll \frac{x}{a^2}\sum_{1 \leq m_2 \leq x^{1/3}}\sum_{1 \leq o_2 \leq x^{1/3}}\frac{1}{m_2o_2}\\
	& \ll \frac{x}{a^2}(\log 2x)^2 \ll_r \frac{x}{a^2}(\log a)^2.\qedhere
\end{align*}
\end{proof}

\begin{lemma}
\label{lem:S13}
Let $ a \in \Z $, $ a \geq 8 $ and $ r > 0 $. If $ 1 \leq x \leq a^r $, then
\[
	\#\left\{\alpha \in Q(a, S_{1,3}(a)) : \Nm(\alpha) \leq x\right\} \ll_r \frac{x}{a^2}(\log a)^2.
\]
The constant implied in the estimate depends only on $ r $.
\end{lemma}
\begin{proof}
Let $ \alpha \in Q(a, S_{1,3}(a)) $, let $ t_1, t_2, t_3 \in \R_{>0} $ be such that $ \alpha = t_1+t_2\rho^2+t_3(\rho')^{-2} $, and let $ \tau(\alpha) = (m, -w, o) $ where $ (m, w, o) \in \Phi(a, S_{1,3}(a)) $. Let $ m = m_0(a, w)+m_1 $, $ o = o_0(a, w)+o_1 $ and $ m_2 = m_1-1 $, $ o_2 = o_1-1 $. Since $ m_1 \geq 2 $ and $ o_1 \geq 2 $, we have $ m_2 \geq 1 $ and $ o_2 \geq 1 $. By \Cref{lem:t1t2}, $ t_1 \geq m_1-1 = m_2 $ and $ t_2 \geq o_1-1 = o_2 $. We also have
\[
	t_3 = \frac{w}{a^2+3a+3} = w_0+\frac{(a+1)w_1}{a^2+3a+3}+\frac{w_2}{a^2+3a+3} \geq \frac{w_2}{2a^2}.
\]

If $ \Nm(\alpha) \leq x $, then by \Cref{lem:Estimates}, $ t_1^3 < x $, $ t_2^3 < x $, and $ a^4t_1t_2t_3 < x $, hence $ m_2^3 < x $, $ o_2^3 < x $, and $ a^2m_2o_2w_2 < 2x $. Let
\[
	M(a, x) := \left\{(m_2, o_2, w) \in \Z_{\geq 1}^3 : m_2^3 < x, o_2^3 < x, w_0 = 0, w_1 = 0, a^2m_2o_2w_2 < 2x\right\}.
\]
We have
\[
	\#\left\{\alpha \in Q(a, S_{1,3}(a)) : \Nm(\alpha) \leq x\right\} \leq \#M(a, x).
\]

The number of triples in $ M(a, x) $ can be estimated as
\begin{align*}
	\#M(a, x)& \leq \sum_{1 \leq m_2 \leq x^{1/3}}\sum_{1 \leq o_2 \leq x^{1/3}}\frac{2x}{a^2m_2o_2} \ll \frac{x}{a^2}\sum_{1 \leq m_2 \leq x^{1/3}}\sum_{1 \leq o_2 \leq x^{1/3}}\frac{1}{m_2o_2}\\
	& \ll \frac{x}{a^2}(\log 2x)^2 \ll_r \frac{x}{a^2}(\log a)^2.\qedhere
\end{align*}
\end{proof}

\begin{prop}
\label{prop:S1}
Let $ a \in \Z $, $ a \geq 8 $ and $ r > 0 $. If $ 1 \leq x \leq a^r $, then
\[
	\#\left\{\alpha \in Q(a, S_1(a)) : \Nm(\alpha) \leq x\right\} \ll_r \frac{x}{a^2}(\log a)^2.
\]
The constant implied in the estimate depends only on $ r $.
\end{prop}
\begin{proof}
We have $ S_1(a) = S_{1,1}(a)\sqcup S_{1,2}(a)\sqcup S_{1,3}(a) $. From the definition of $ Q(a, S_1(a)) $, we get $ Q(a, S_1(a)) = Q(a, S_{1,1}(a))\sqcup Q(a, S_{1,2}(a))\sqcup Q(a, S_{1,3}(a)) $. The proposition now follows from \Cref{lem:S11,lem:S12,lem:S13}.
\end{proof}

\subsection{Estimate for \texorpdfstring{$ Q(a, S_2(a)) $}{Q(a,S2(a))}}
\label{sec:S2}

We recall that
\[
	S_2(a) = \left\{(m_1, o_1, w_0, w_1, w_2) \in \Z^2\times W_1(a) : m_1\geq 2, o_1 = 1\right\}.
\]
We write $ S_2(a) = S_{2,1}(a)\sqcup S_{2,2}(a)\sqcup \dots \sqcup S_{2,6}(a) $, where
\begin{align*}
	S_{2,1}(a)& := \left\{(m_1, o_1, w_0, w_1, w_2) \in S_2(a) : w_0 \geq 1, w_1 \geq 0, w_2 \geq 1\right\},\\
	S_{2,2}(a)& := \left\{(m_1, o_1, w_0, w_1, w_2) \in S_2(a) : w_0 = 0, w_1 \geq 1, w_2 \geq 1\right\},\\
	S_{2,3}(a)& := \left\{(m_1, o_1, w_0, w_1, w_2) \in S_2(a) : w_0 = 0, w_1 = 0, w_2 \geq 1\right\},\\
	S_{2,4}(a)& := \left\{(m_1, o_1, w_0, w_1, w_2) \in S_2(a) : w_0 \geq 1, w_1 \geq 1, w_2 = 0\right\},\\
	S_{2,5}(a)& := \left\{(m_1, o_1, w_0, w_1, w_2) \in S_2(a) : w_0 = 0, w_1 \geq 1, w_2 = 0\right\},\\
	S_{2,6}(a)& := \left\{(m_1, o_1, w_0, w_1, w_2) \in S_2(a) : w_0 \geq 1, w_1 = 0, w_2 = 0\right\}.
\end{align*}

\begin{lemma}
\label{lem:Estimatet2}
Let $ a \in \Z_{\geq 1} $. Let $ w \in \Z_{\geq 1} $, $ m \in \Z $, and $ o = o_0(a, w)+1 $. If $ \alpha = t_1+t_2\rho^2+t_3(\rho')^{-2} $ and $ \tau(\alpha) = (m, -w, o) $, then
\[
	t_2 = \begin{cases}
		1,&w_1 = w_2 = 0,\\
		1-\frac{(a+2)w_2-w_1}{a^2+3a+3},&w_2 \geq 1,\\
		\frac{w_1}{a^2+3a+3},&w_1 \geq 1, w_2 = 0.
	\end{cases}
\]
\end{lemma}
\begin{proof}
We have
\[
	t_2 = o-\frac{a+2}{a^2+3a+3}w = o_0(a,w)+1-\frac{a+2}{a^2+3a+3}w = 1-\left(\frac{a+2}{a^2+3a+3}w-\left\lfloor\frac{a+2}{a^2+3a+3}w\right\rfloor\right).
\]
By \Cref{lem:W2},
\[
	\frac{a+2}{a^2+3a+3}w-\left\lfloor\frac{a+2}{a^2+3a+3}w\right\rfloor = \begin{cases}
		0,&w_1 = w_2 = 0,\\
		\frac{(a+2)w_2-w_1}{a^2+3a+3},&w_2 \geq 1,\\
		1-\frac{w_1}{a^2+3a+3},&w_1 \geq 1, w_2 = 0.
	\end{cases}
\]
The lemma follows.
\end{proof}

\begin{lemma}
\label{lem:S21}
Let $ a \in  \Z $, $ a \geq 8 $, and $ r > 0 $. If $ 1 \leq x \leq a^r $, then
\[
	\#\left\{\alpha \in Q(a, S_{2,1}(a)) : \Nm(\alpha) \leq x\right\} \ll_r \frac{x}{a^2}(\log a)^2.
\]
The constant implied in the estimate depends only on $ r $.
\end{lemma}
\begin{proof}
Let $ \alpha \in Q(a, S_{2,1}(a)) $, let $ t_1, t_2, t_3 \in \R_{>0} $ be such that $ \alpha = t_1+t_2\rho^2+t_3(\rho')^{-2} $, and let $ \tau(\alpha) = (m,-w,o) $, where $ (m, w, o) \in \Phi(a, S_{2,1}(a)) $. Let $ m = m_0(a, w)+m_1 $, $ o = o_0(a, w)+1 $, and $ m_2 = m_1-1 $. Since $ m_1 \geq 2 $, we have $ m_2 \geq 1 $.

By \Cref{lem:t1t2}, $ t_1 \geq m_1-1 = m_2 $.

We have $ w_2 \geq 1 $, hence by \Cref{lem:Estimatet2},
\[
	t_2 = 1-\frac{(a+2)w_2-w_1}{a^2+3a+3}.
\]

We also have
\[
	t_3 = \frac{w}{a^2+3a+3} = w_0+\frac{(a+1)w_1}{a^2+3a+3}+\frac{w_2}{a^2+3a+3} \geq w_0.
\]

First, assume $ w_2 \leq \frac{a}{2} $. Then
\[
	t_2 > 1-\frac{(a+2)w_2}{a^2+3a+3} \geq \frac{\frac{a^2}{2}+2a+3}{a^2+3a+3} > \frac{1}{2}.
\]

If $ \Nm(\alpha) \leq x $, then by \Cref{lem:Estimates}, $ t_1^3 < x $ and $ a^4t_1t_2t_3 < x $, hence $ m_2^3 < x $ and $ a^4m_2w_0 < 2x $. Let
\[
	M_1(a, x) := \left\{(m_2, w) \in \Z_{\geq 1}^2 : m_2^3 < x, w_0 \geq 1, a^4m_2w_0 < 2x, 1 \leq w_2 \leq \frac{a}{2}\right\}.
\]

The number of tuples in $ M_1(a, x) $ can be estimated as
\begin{align*}
	\#M_1(a, x)& \leq \sum_{1 \leq m_2 \leq x^{1/3}}\sum_{0 \leq w_1 \leq a+1}\sum_{1 \leq w_2 \leq a/2}\frac{2x}{a^4m_2} \ll \frac{x}{a^2}\sum_{1 \leq m_2 \leq x^{1/3}}\frac{1}{m_2}\\
	& \ll \frac{x}{a^2}(\log 2x) \ll_r \frac{x}{a^2}(\log a).
\end{align*}

Secondly, assume $ w_2 > \frac{a}{2} $. We write $ w_2 = (a+1)-z_2 $, where $ 1 \leq z_2 < \frac{a}{2}+1 $. Now
\[
	t_2 \geq 1-\frac{a+2}{a^2+3a+3}w_2 > \frac{(a+2)z_2}{a^2+3a+3} > \frac{z_2}{2a}.
\]

If $ \Nm(\alpha) \leq x $, then by \Cref{lem:Estimates}, $ t_1^3 < x $ and $ a^4t_1t_2t_3 < x $, hence $ m_2^3 < x $ and $ a^3m_2z_2w_0 < 2x $. Let
\[
	M_2(a, x) := \left\{(m_2, w) \in \Z_{\geq 1}^2 : m_2^3 < x, w_0 \geq 1, w_2 = (a+1)-z_2, 1 \leq z_2 < \frac{a}{2}+1, a^3m_2z_2w_0 < 2x\right\}.
\]

The number of tuples in $ M_2(a, x) $ can be estimated as
\begin{align*}
	\#M_2(a, x)& \leq \sum_{1 \leq m_2 \leq x^{1/3}}\sum_{0 \leq w_1 \leq a+1}\sum_{1 \leq z_2 < \frac{a}{2}+1}\frac{2x}{a^3m_2z_2} \ll \frac{x}{a^2}(\log a)\sum_{1 \leq m_2 \leq x^{1/3}}\frac{1}{m_2}\\
	& \ll \frac{x}{a^2}(\log a)(\log 2x) \ll_r \frac{x}{a^2}(\log a)^2.
\end{align*}
We have
\[
	\#\left\{\alpha \in Q(a, S_{2,1}(a)) : \Nm(\alpha)\leq x\right\} \leq \#M_1(a, x)+\#M_2(a, x) \ll_r \frac{x}{a^2}(\log a)^2.\qedhere
\]
\end{proof}

\begin{lemma}
\label{lem:S22}
Let $ a \in \Z $, $ a \geq 8 $ and $ r > 0 $. If $ 1 \leq x \leq a^r $, then
\[
	\#\left\{\alpha \in Q(a, S_{2,2}(a)) : \Nm(\alpha)\leq x\right\} \ll_r \frac{x}{a^2}(\log a)^2.
\]
The constant implied in the estimate depends only on $ r $.
\end{lemma}
\begin{proof}
Let $ \alpha \in Q(a, S_{2,2}(a)) $, let $ t_1, t_2, t_3 \in \R_{>0} $ be such that $ \alpha = t_1+t_2\rho^2+t_3(\rho')^{-2} $, and let $ \tau(\alpha) = (m, -w, o) $, where $ (m, w, o) \in \Phi(a, S_{2,2}(a)) $. Let $ m = m_0(a, w)+m_1 $, $ o = o_0(a, w)+1 $, and $ m_2 = m_1-1 $. Since $ m_1 \geq 2 $, we have $ m_2 \geq 1 $.

By \Cref{lem:t1t2}, $ t_1 \geq m_1-1 = m_2 $.

We have $ w_2 \geq 1 $, hence by \Cref{lem:Estimatet2},
\[
	t_2 = 1-\frac{(a+2)w_2-w_1}{a^2+3a+3}.
\]

We also have
\[
	t_3 = \frac{w}{a^2+3a+3} = w_0+\frac{(a+1)w_1}{a^2+3a+3}+\frac{w_2}{a^2+3a+3} \geq \frac{(a+1)w_1}{a^2+3a+3} \geq \frac{w_1}{2a}.
\]

First, assume $ w_2 \leq \frac{a}{2} $. Then
\[
	t_2 > 1-\frac{(a+2)w_2}{a^2+3a+3} \geq \frac{\frac{a^2}{2}+2a+3}{a^2+3a+3} > \frac{1}{2}.
\]
If $ \Nm(\alpha) \leq x $, then by \Cref{lem:Estimates}, $ t_1^3 < x $ and $ a^4t_1t_2t_3 < x $, hence $ m_2^3 < x $ and $ a^3m_2w_1 < 4x $. Let
\[
	M_1(a, x) := \left\{(m_2, w) \in \Z_{\geq 1}^2 : m_2^3 < x, w_0 = 0, w_1 \geq 1, a^3m_2w_1 < 4x, 1 \leq w_2 \leq \frac{a}{2}\right\}.
\]
The number of tuples in $ M_1(a, x) $ can be estimated as
\begin{align*}
	\#M_1(a, x)& \leq \sum_{1 \leq m_2 \leq x^{1/3}}\sum_{1 \leq w_2 \leq \frac{a}{2}}\frac{4x}{a^3m_2} \ll \frac{x}{a^2}\sum_{1 \leq m_2 \leq x^{1/3}}\frac{1}{m_2}\\
	& \ll \frac{x}{a^2}(\log 2x) \ll_r \frac{x}{a^2}(\log a).
\end{align*}

Secondly, assume $ w_2 > \frac{a}{2} $. We write $ w_2 = (a+1)-z_2 $, where $ 1 \leq z_2 < \frac{a}{2}+1 $. Now
\[
	t_2 \geq 1-\frac{a+2}{a^2+3a+3}w_2 > \frac{(a+2)z_2}{a^2+3a+3} > \frac{z_2}{2a}.
\]
If $ \Nm(\alpha) \leq x $, then by \Cref{lem:Estimates}, $ t_1^3 < x $ and $ a^4t_1t_2t_3 < x $, hence $ m_2^3 < x $ and $ a^2m_2w_1z_2 < 4x $. Let
\[
	M_2(a, x) := \left\{(m_2, w) \in \Z_{\geq 1}^2 : m_2^3 < x, w_0 = 0, w_1 \geq 1, w_2 = (a+1)-z_2, 1 \leq z_2 < \frac{a}{2}+1, a^2m_2w_1z_2 < 4x\right\}.
\]
The number of tuples in $ M_2(a, x) $ can be estimated as
\begin{align*}
	\#M_2(a, x)& \leq \sum_{1 \leq m_2 \leq x^{1/3}}\sum_{1 \leq z_2 < \frac{a}{2}+1}\frac{4x}{a^2m_2z_2} \ll \frac{x}{a^2}(\log a)\sum_{1 \leq m_2 \leq x^{1/3}}\frac{1}{m_2}\\ & \ll \frac{x}{a^2}(\log a)(\log 2x) \ll_r \frac{x}{a^2}(\log a)^2.
\end{align*}

We have
\[
	\#\left\{\alpha \in Q(a, S_{2,2}(a)) : \Nm(\alpha)\leq x\right\} \leq \#M_1(a, x)+\#M_2(a, x) \ll_r \frac{x}{a^2}(\log a)^2.\qedhere
\]
\end{proof}

\begin{lemma}
\label{lem:S23}
Let $ a \in \Z $, $ a \geq 8 $ and $ r > 0 $. If $ 1 \leq x \leq a^r $, then
\[
	\#\left\{\alpha \in Q(a, S_{2,3}(a)) : \Nm(\alpha) \leq x\right\} \ll_r \frac{x}{a^2}(\log a).
\]
The constant implied in the estimate depends only on $ r $.
\end{lemma}
\begin{proof}
Let $ \alpha \in Q(a, S_{2,3}(a)) $, let $ t_1, t_2, t_3 \in \R_{>0} $ be such that $ \alpha = t_1+t_2\rho^2+t_3(\rho')^{-2} $, and let $ \tau(\alpha) = (m, -w, o) $, where $ (m, w, o) \in \Phi(a, S_{2,3}(a)) $. Let $ m = m_0(a, w)+m_1 $, $ o = o_0(a, w)+1 $, and $ m_2 = m_1-1 $. Since $ m_1 \geq 2 $, we have $ m_2 \geq 1 $.

By \Cref{lem:t1t2}, $ t_1 \geq m_1-1 = m_2 $.

We have $ w_2 \geq 1 $, hence by \Cref{lem:Estimatet2},
\[
	t_2 = 1-\frac{(a+2)w_2-w_1}{a^2+3a+3} = 1-\frac{(a+2)w_2}{a^2+3a+3}.
\]

We also have
\[
	t_3 = \frac{w}{a^2+3a+3} = w_0+\frac{(a+1)w_1}{a^2+3a+3}+\frac{w_2}{a^2+3a+3} = \frac{w_2}{a^2+3a+3} \geq \frac{w_2}{2a^2}.
\]

First, assume $ w_2 \leq \frac{a}{2} $. Then
\[
	t_2 \geq 1-\frac{(a+2)w_2}{a^2+3a+3} \geq \frac{\frac{a^2}{2}+2a+3}{a^2+3a+3} > \frac{1}{2}.
\]

If $ \Nm(\alpha) \leq x $, then by \Cref{lem:Estimates}, $ t_1^3 < x $ and $ a^4t_1t_2t_3 < x $, hence $ m_2^3 < x $ and $ a^2m_2w_2 < 4x $. Let
\[
	M_1(a, x) := \left\{(m_2, w_2) \in \Z_{\geq 1}^2 : m_2^3 < x, a^2m_2w_2 < 4x\right\}.
\]
The number of tuples in $ M_1(a, x) $ can be estimated as
\[
	\#M_1(a, x) \leq \sum_{1 \leq m_2 \leq x^{1/3}}\frac{4x}{a^2m_2} \ll \frac{x}{a^2}(\log 2x) \ll_r \frac{x}{a^2}(\log a).
\]

Secondly, assume $ w_2 > \frac{a}{2} $. We write $ w_2 = (a+1)-z_2 $, where $ 1 \leq z_2 < \frac{a}{2}+1 $. Now
\[
	t_2 > 1-\frac{a+2}{a^2+3a+3}w_2 > \frac{(a+2)z_2}{a^2+3a+3} > \frac{z_2}{2a}.
\]
If $ \Nm(\alpha) \leq x $, then by \Cref{lem:Estimates}, $ t_1^3 < x $ and $ a^4t_1t_2t_3 < x $, hence $ m_2^3 < x $ and using $ t_3 \geq \frac{w_2}{2a^2} > \frac{1}{4a} $, we get $ a^2m_2z_2 < 8x $. Let
\[
	M_2(a, x) := \left\{(m_2, z_2) \in \Z_{\geq 1}^2 : m_2^3 < x, z_2 < \frac{a}{2}+1, a^2m_2z_2 < 8x\right\}.
\]
The number of tuples in $ M_2(a, x) $ can be estimated as
\[
	\#M_2(a, x) \ll_r \frac{x}{a^2}(\log a).
\]

We have
\[
	\#\left\{\alpha \in Q(a, S_{2,3}(a)) : \Nm(\alpha)\leq x\right\} \leq \#M_1(a, x)+\#M_2(a, x) \ll_r \frac{x}{a^2}(\log a).\qedhere
\]
\end{proof}

\begin{lemma}
\label{lem:S24}
Let $ a \in \Z $, $ a \geq 8 $ and $ r > 0 $. If $ 1 \leq x \leq a^r $, then
\[
	\#\left\{\alpha \in Q(a, S_{2,4}(a)) : \Nm(\alpha) \leq x\right\} \ll_r \frac{x}{a^2}(\log a)^2.
\]
The constant implied in the estimate depends only on $ r $.
\end{lemma}
\begin{proof}
Let $ \alpha \in Q(a, S_{2,4}(a)) $, let $ t_1, t_2, t_3 \in \R_{>0} $ be such that $ \alpha = t_1+t_2\rho^2+t_3(\rho')^{-2} $, and let $ \tau(\alpha) = (m, -w, o) $ where $ (m, w, o) \in \Phi(a, S_{2,4}(a)) $. Let $ m = m_0(a, w)+m_1 $, $ o = o_0(a, w)+1 $, and $ m_2 = m_1-1 $. Since $ m_1 \geq 2 $, we have $ m_2 \geq 1 $.

By \Cref{lem:t1t2}, $ t_1 \geq m_1-1 = m_2 $.

We have $ w_1 \geq 1 $ and $ w_2 = 0 $, hence by \Cref{lem:Estimatet2},
\[
	t_2 = \frac{w_1}{a^2+3a+3} \geq \frac{w_1}{2a^2}.
\]

We also have
\[
	t_3 = \frac{w}{a^2+3a+3} = w_0+\frac{(a+1)w_1}{a^2+3a+3}+\frac{w_2}{a^2+3a+3} \geq w_0.
\]

If $ \Nm(\alpha) \leq x $, then by \Cref{lem:Estimates}, $ t_1^3 < x $ and $ a^4t_1t_2t_3 < x $, hence $ m_2^3 < x $ and $ a^2m_2w_0w_1 < 2x $. Let
\[
	M(a, x) := \left\{(m_2, w) \in \Z_{\geq 1}^2 : m_2^3 < x, w_0 \geq 1, w_1 \geq 1, w_2 = 0, a^2m_2w_0w_1 < 2x\right\}.
\]
The number of tuples in $ M(a, x) $ can be estimated as
\begin{align*}
	\#M(a, x)& \leq \sum_{1 \leq m_2 \leq x^{1/3}}\sum_{1 \leq w_1 \leq a+2}\frac{2x}{a^2m_2w_1} \ll \frac{x}{a^2}\sum_{1 \leq m_2 \leq x^{1/3}}\sum_{1 \leq w_1 \leq a+2}\frac{1}{m_2w_1}\\
	& \ll \frac{x}{a^2}(\log 2x)(\log a) \ll_r \frac{x}{a^2}(\log a)^2.
\end{align*}

We have
\[
	\#\left\{\alpha \in Q(a, S_{2,4}(a)) : \Nm(\alpha) \leq x\right\} \leq \#M(a, x) \ll_r \frac{x}{a^2}(\log a)^2.\qedhere
\]
\end{proof}

\begin{lemma}
\label{lem:S25}
Let $ a \in \Z $, $ a \geq 8 $. If $ 1 \leq x $, then
\[
	\#\left\{\alpha \in Q(a, S_{2,5}(a)) : \Nm(\alpha) \leq x\right\} \ll \left(\frac{x}{a}\right)^{2/3}.
\]
\end{lemma}
\begin{proof}
Let $ \alpha \in Q(a, S_{2,5}(a)) $, let $ t_1, t_2, t_3 \in \R_{>0} $ be such that $ \alpha = t_1+t_2\rho^2+t_3(\rho')^{-2} $, and let $ \tau(\alpha) = (m, -w, o) $, where $ (m, w, o) \in \Phi(a, S_{2,5}(a)) $. Let $ m = m_0(a, w)+m_1 $, $ o = o_0(a, w)+1 $, and $ m_2 = m_1-1 $. Since $ m_1 \geq 2 $, we have $ m_2 \geq 1 $.

By \Cref{lem:t1t2}, $ t_1 \geq m_1-1 = m_2 $.

We have $ w_1 \geq 1 $ and $ w_2 = 0 $, hence by \Cref{lem:Estimatet2},
\[
	t_2 = \frac{w_1}{a^2+3a+3} \geq \frac{w_1}{2a^2}.
\]

We also have
\[
	t_3 = \frac{w}{a^2+3a+3} = w_0+\frac{(a+1)w_1}{a^2+3a+3}+\frac{w_2}{a^2+3a+3} = \frac{(a+1)w_1}{a^2+3a+3} \geq \frac{w_1}{2a}.
\]

By \Cref{lem:Estimates}, $ a^2t_1^2t_3 < x $ and $ a^4t_1t_2t_3 < x $, hence $ am_2^2w_1 < 2x $ and $ am_2w_1^2 < 4x $. Let
\[
	M(a, x) := \#\left\{(m_2, w_1) \in \Z_{\geq 1}^2 : am_2^2w_1 < 2x, am_2w_1^2 < 4x\right\}.
\]

We further split $ M(a, x) $ into two sets
\begin{align*}
	M_1(a, x) := \#\left\{(m_2, w_1) \in M(a, x) : m_2 \leq w_1\right\},\\
	M_2(a, x) := \#\left\{(m_2, w_1) \in M(a, x) : m_2 > w_1\right\}.
\end{align*}
If $ (m_2, w_1) \in M_1(a, x) $, then $ am_2^3 < 2x $, hence $ m_2 < (2x/a)^{1/3} $ and
\[
	\#M_1(a, x) \leq \sum_{1 \leq m_2 < (2x/a)^{1/3}}\left(\frac{4x}{am_2}\right)^{1/2} \ll \left(\frac{x}{a}\right)^{1/2}\sum_{1 \leq m_2 < (2x/a)^{1/3}}m_2^{-1/2} \ll \left(\frac{x}{a}\right)^{1/2}\left(\frac{2x}{a}\right)^{1/6} \ll \left(\frac{x}{a}\right)^{2/3}.
\]

If $ (m_2, w_1) \in M_2(a, x) $, then $ aw_1^3 < 4x $, hence $ w_1 < (4x/a)^{1/3} $ and similarly as in the case of $ M_1(a, x) $ we get
\[
	\#M_2(a, x) \leq \left(\frac{x}{a}\right)^{2/3}.
\]

Finally,
\[
	\#\left\{\alpha \in Q(a, S_{2,5}(a)) : \Nm(\alpha) \leq x\right\} \leq \#M(a, x) = \#M_1(a, x)+\#M_2(a, x) \ll \left(\frac{x}{a}\right)^{2/3}.\qedhere
\]
\end{proof}

\begin{lemma}
\label{lem:S26}
Let $ a \in \Z $, $ a \geq 8 $ and $ r > 0 $. If $ 1 \leq x \leq a^r $, then
\[
	\#\left\{\alpha \in Q(a, S_{2,6}(a)) : \Nm(\alpha) \leq x\right\} \ll_r \frac{x}{a^4}(\log a).
\]
The constant implied in the estimate depends only on $ r $.
\end{lemma}
\begin{proof}
Let $ \alpha \in Q(a, S_{2,6}(a)) $, let $ t_1, t_2, t_3 \in \R_{>0} $ be such that $ \alpha = t_1+t_2\rho^2+t_3(\rho')^{-2} $, and let $ \tau(\alpha) = (m, -w, o) $, where $ (m, w, o) \in \Phi(a, S_{2,6}(a)) $. Let $ m = m_0(a, w)+m_1 $, $ o = o_0(a, w)+1 $, and $ m_2 = m_1-1 $. Since $ m_1 \geq 2 $, we have $ m_2 \geq 1 $.

By \Cref{lem:t1t2}, $ t_1 \geq m_1-1 = m_2 $.

We have $ w_1 = 0 $ and $ w_2 = 0 $, hence by \Cref{lem:Estimatet2}, $ t_2 = 1 $. We also have
\[
	t_3 = \frac{w}{a^2+3a+3} = w_0+\frac{(a+1)w_1}{a^2+3a+3}+\frac{w_2}{a^2+3a+3} \geq w_0.
\]

By \Cref{lem:Estimates}, $ t_1^3 < x $ and $ a^4t_1t_2t_3 < x $, hence $ m_2^3 < x $ and $ a^4m_2w_0 < x $. Let
\[
	M(a, x) := \left\{(m_2, w_0) \in \Z_{\geq 1}^2 : m_2^3 < x, a^4m_2w_0 < x\right\}.
\]

The number of pairs in $ M(a, x) $ can be estimated as
\[
	\#M(a, x) \leq \sum_{1 \leq m_2 < x^{1/3}}\frac{x}{a^4m_2} \ll \frac{x}{a^4}(\log 2x) \ll_r \frac{x}{a^4}(\log a).
\]

We have
\[
	\#\left\{\alpha \in Q(a, S_{2,6}(a)) : \Nm(\alpha) \leq x\right\} \leq \#M(a,x) \ll_r \frac{x}{a^4}(\log a).\qedhere
\]
\end{proof}

\begin{prop}
\label{prop:S2}
Let $ a \in \Z $, $ a \geq 8 $ and $ r > 0 $. If $ 1 \leq x \leq a^r $, then
\[
	\#\left\{\alpha \in Q(a, S_2(a)) : \Nm(\alpha) \leq x\right\} \ll_r \frac{x}{a^2}(\log a)^2+\left(\frac{x}{a}\right)^{2/3}.
\]
The constant implied in the estimate depends only on $ r $.
\end{prop}
\begin{proof}
We have $ S_2(a) = S_{2,1}(a)\sqcup S_{2,2}(a)\sqcup \dots \sqcup S_{2,6}(a) $, hence $ Q(a, S_2(a)) = Q(a, S_{2,1}(a))\sqcup Q(a, S_{2,2}(a))\sqcup \dots\sqcup Q(a, S_{2,6}(a)) $. The proposition now follows from \Crefrange{lem:S21}{lem:S26}.
\end{proof}

\subsection{Estimate for \texorpdfstring{$ R(a, S_3(a)) $}{R(a,S3(a))}}
\label{sec:S3}

We recall that
\[
	S_3(a) = \left\{(m_1, o_1, w_0, w_1, w_2) \in \Z^2 \times W_2(a) : m_1 = 1, o_1 \geq 2\right\}.
\]

We write $ S_3(a) = S_{3,1}(a)\sqcup S_{3,2}(a)\sqcup \dots\sqcup S_{3,5}(a) $, where
\begin{align*}
	S_{3,1}(a)& := \left\{(m_1, o_1, w_0, w_1, w_2) \in S_3(a) : w_0 \geq 1, w_1 \geq 0, w_2 \geq 1\right\},\\
	S_{3,2}(a) & := \left\{(m_1, o_1, w_0, w_1, w_2) \in S_3(a) : w_0 = 0, w_1 \geq 1, w_2 \geq 1\right\},\\
	S_{3,3}(a)& := \left\{(m_1, o_1, w_0, w_1, w_2) \in S_3(a) : w_0 = 0, w_1 = 0, w_2 \geq 1\right\},\\
	S_{3,4}(a)& := \left\{(m_1, o_1, w_0, w_1, w_2) \in S_3(a) : w_0 \geq 1, w_1 \geq 0, w_2 = 0\right\},\\
	S_{3,5}(a)& := \left\{(m_1, o_1, w_0, w_1, w_2) \in S_3(a) : w_0 = 0, w_1 \geq 1, w_2 = 0\right\}.
\end{align*}

\begin{lemma}
\label{lem:S3t1t2}
Let $ a \in \Z_{\geq 1} $, $ w \in \Z_{\geq 1} $, $ m = m_0(a,w)+1 $, $ o \in \Z $, and $ (w_0, w_1, w_2) = \psi_a(w) $. If $ \alpha = t_1+t_2\rho^2+t_3(\rho')^{-2} $ is such that $ \tau(\alpha) = (m, -w, o) $, then
\[
	t_1 = \begin{cases}
		1-\frac{w_1}{a^2+3a+3},&\text{if }w_2 = 0,\\
		\frac{(a+1)w_2-w_1}{a^2+3a+3},&\text{if }w_2 \geq 1.
	\end{cases}
\]
\end{lemma}
\begin{proof}
We have
\[
	t_1 = m+\frac{a+1}{a^2+3a+3}w = m_0+1+\frac{a+1}{a^2+3a+3}w = 1-\left(-\frac{a+1}{a^2+3a+3}w-\left\lfloor-\frac{a+1}{a^2+3a+3}w\right\rfloor\right).
\]

By \Cref{lem:W3},
\[
	-\frac{a+1}{a^2+3a+3}w-\left\lfloor-\frac{a+1}{a^2+3a+3}w\right\rfloor = \begin{cases}
		\frac{w_1}{a^2+3a+3},&w_2 = 0,\\
		1+\frac{w_1-(a+1)w_2}{a^2+3a+3},&w_2 \geq 1
	\end{cases}
\]
and the expression for $ t_1 $ follows.
\end{proof}

\begin{lemma}
\label{lem:S31}
Let $ a \in \Z $, $ a \geq 8 $ and $ r > 0 $. If $ 1 \leq x \leq a^r $, then
\[
	\#\left\{\alpha \in R(a, S_{3,1}(a)) : \Nm(\alpha) \leq x\right\} \ll_r \frac{x}{a^2}(\log a)^2.
\]
The constant implied in the estimate depends only on $ r $.
\end{lemma}
\begin{proof}
Let $ \alpha \in R(a, S_{3,1}(a)) $, let $ t_1, t_2, t_3 \in \R_{>0} $ be such that $ \alpha = t_1+t_2\rho^2+t_3(\rho')^{-2} $, and let $ \tau(\alpha) = (m, -w, o) $, where $ (m, w, o) \in \Psi(a, S_{3,1}(a)) $. Let $ m = m_0(a, w)+1 $, $ o = o_0(a, w)+o_1 $, and $ o_2 = o_1-1 $. Since $ o_1 \geq 2 $, we have $ o_2 \geq 1 $. Since $ w_2 \geq 1 $, by \Cref{lem:S3t1t2},
\[
	t_1 = \frac{(a+1)w_2-w_1}{a^2+3a+3}.
\]
By \Cref{lem:t1t2}, $ t_2 \geq o_1-1 = o_2 $. We also have
\[
	t_3 = \frac{w}{a^2+3a+3} = w_0+\frac{(a+2)w_1}{a^2+3a+3}+\frac{w_2}{a^2+3a+3} \geq w_0.
\]

We note that $ 0 \leq w_1 \leq a $ because $ w_2 \geq 1 $. Let us distinguish two cases. First, if $ w_2 \geq 2 $, then we let $ w_2' = w_2-1 $. We have
\[
	t_1 = \frac{(a+1)(1+w_2')-w_1}{a^2+3a+3} = \frac{a+1-w_1}{a^2+3a+3}+\frac{(a+1)w_2'}{a^2+3a+3} \geq \frac{w_2'}{2a}.
\]
If $ \Nm(\alpha) \leq x $, then by \Cref{lem:Estimates}, $ t_2^3 < x $ and $ a^4t_1t_2t_3 < x $, hence $ o_2^3 < x $ and $ a^3o_2w_0w_2' < 2x $. Let
\[
	M_1(a, x) := \left\{(o_2, w) \in \Z_{\geq 1}^2 : o_2^3 < x, w_0 \geq 1, w_2 = w_2'+1, 1 \leq w_2' \leq a, a^3o_2w_0w_2' < 2x\right\}.
\]
The number of pairs in $ M_1(a, x) $ can be estimated as
\begin{align*}
	\#M_1(a, x)& \leq \sum_{1 \leq o_2 < x^{1/3}}\sum_{0 \leq w_1 \leq a}\sum_{1 \leq w_2' \leq a}\frac{2x}{a^3o_2w_2'} \ll \frac{x}{a^2}\sum_{1 \leq o_2 < x^{1/3}}\sum_{1 \leq w_2' \leq a}\frac{1}{o_2w_2'}\\
	& \ll \frac{x}{a^2}(\log 2x)(\log a) \ll_r \frac{x}{a^2}(\log a)^2.
\end{align*}

Secondly, if $ w_2 = 1 $, then we let $ w_1 = (a+1)-z_1 $, where $ 1 \leq z_1 \leq a+1 $. We have
\[
	t_1 = \frac{(a+1)-w_1}{a^2+3a+3} = \frac{z_1}{a^2+3a+3} \geq \frac{z_1}{2a^2}.
\]
If $ \Nm(\alpha) \leq x $, then by \Cref{lem:Estimates}, $ t_2^3 < x $ and $ a^4t_1t_2t_3 < x $, hence $ o_2^3 < x $ and $ a^2o_2w_0z_1 < 2x $. Let
\[
	M_2(a, x) := \left\{(o_2, w) \in \Z_{\geq 1}^2 : o_2^3 < x, w_0 \geq 1, w_1 = (a+1)-z_1, 1 \leq z_1 \leq a+1, w_2 = 1, a^2o_2w_0z_1 < 2x\right\}.
\]
The number of pairs in $ M_2(a, x) $ can be estimated as
\begin{align*}
	\#M_2(a, x)& \leq \sum_{1 \leq o_2 < x^{1/3}}\sum_{1 \leq z_1 \leq a+1}\frac{2x}{a^2o_2z_1} \ll \frac{x}{a^2}\sum_{1 \leq o_2 < x^{1/3}}\sum_{1 \leq z_1 \leq a+1}\frac{1}{o_2z_1}\\
	& \ll \frac{x}{a^2}(\log 2x)(\log a) \ll_r \frac{x}{a^2}(\log a)^2.
\end{align*}

Finally,
\[
	\#\left\{\alpha \in R(a, S_{3,1}(a)) : \Nm(\alpha) \leq x\right\} \leq \#M_1(a, x)+\#M_2(a, x) \ll_r \frac{x}{a^2}(\log a)^2.\qedhere
\]
\end{proof}

\begin{lemma}
\label{lem:S32}
Let $ a \in \Z $, $ a \geq 8 $ and $ r > 0 $. If $ 1 \leq x \leq a^r $, then
\[
	\#\left\{\alpha \in R(a, S_{3,2}(a)) : \Nm(\alpha) \leq x\right\} \ll_r \frac{x}{a^2}(\log a)^2.
\]
The constant implied in the estimate depends only on $ r $.
\end{lemma}
\begin{proof}
Let $ \alpha \in R(a, S_{3,2}(a)) $, let $ t_1, t_2, t_3 \in \R_{>0} $ be such that $ \alpha = t_1+t_2\rho^2+t_3(\rho')^{-2} $, and let $ \tau(\alpha) = (m, -w, o) $, where $ (m, w, o) \in \Psi(a, S_{3,2}(a)) $. Let $ m = m_0(a, w)+1 $, $ o = o_0(a, w)+o_1 $, and $ o_2 = o_1-1 $. Since $ o_1 \geq 2 $, we have $ o_2 \geq 1 $. Since $ w_2 \geq 1 $, by \Cref{lem:S3t1t2},
\[
	t_1 = \frac{(a+1)w_2-w_1}{a^2+3a+3}.
\]
By \Cref{lem:t1t2}, $ t_2 \geq o_1-1 = o_2 $. We also have
\[
	t_3 = \frac{w}{a^2+3a+3} = w_0+\frac{(a+2)w_1}{a^2+3a+3}+\frac{w_2}{a^2+3a+3} \geq \frac{(a+2)w_1}{a^2+3a+3} \geq \frac{w_1}{2a}.
\]

Let us distinguish two cases. First, if $ w_2 \geq 2 $, then we let $ w_2' = w_2-1 $. We have
\[
	t_1 = \frac{(a+1)(1+w_2')-w_1}{a^2+3a+3} = \frac{a+1-w_1}{a^2+3a+3}+\frac{(a+1)w_2'}{a^2+3a+3} \geq \frac{w_2'}{2a}.
\]
If $ \Nm(\alpha) \leq x $, then by \Cref{lem:Estimates}, $ t_2^3 < x $ and $ a^4t_1t_2t_3 < x $, hence $ o_2^3 < x $ and $ a^2o_2w_1w_2' < 4x $. Let
\[
	M_1(a, x) := \left\{(o_2, w) \in \Z_{\geq 1}^2 : o_2^3 < x, w_0 = 0, w_1 \geq 1, w_2 = w_2'+1, 1 \leq w_2' \leq a, a^2o_2w_1w_2' < 4x\right\}.
\]
The number of pairs in $ M_1(a, x) $ can be estimated as
\begin{align*}
	\#M_1(a, x)& \leq \sum_{1 \leq o_2 < x^{1/3}}\sum_{1 \leq w_2' \leq a}\frac{4x}{a^2o_2w_2'} \ll \frac{x}{a^2}\sum_{1 \leq o_2 < x^{1/3}}\sum_{1 \leq w_2' \leq a}\frac{1}{o_2w_2'}\\
	& \ll \frac{x}{a^2}(\log 2x)(\log a) \ll_r \frac{x}{a^2}(\log a)^2.
\end{align*}

Secondly, if $ w_2 = 1 $, then we further distinguish two cases. In the case $ w_1 \leq \frac{a}{2} $, we get
\[
	t_1 = \frac{(a+1)-w_1}{a^2+3a+3} > \frac{a+1}{2(a^2+3a+3)} \geq \frac{1}{4a}.
\]

If $ \Nm(\alpha) \leq x $, then by \Cref{lem:Estimates}, $ t_2^3 < x $ and $ a^4t_1t_2t_3 < x $, hence $ o_2^3 < x $ and $ a^2o_2w_1 < 8x $. Let
\[
	M_2(a, x) := \left\{(o_2, w) \in \Z_{\geq 1}^2 : o_2^3 < x, w_0 = 0, 1 \leq w_1 \leq \frac{a}{2}, w_2 = 1, a^2o_2w_1 < 8x\right\}.
\]
The number of pairs in $ M_2(a, x) $ can be estimated as
\[
	\#M_2(a, x) \leq \sum_{1 \leq o_2 < x^{1/3}}\frac{8x}{a^2o_2} \ll \frac{x}{a^2}\sum_{1 \leq o_2 < x^{1/3}}\frac{1}{o_2} \ll \frac{x}{a^2}(\log 2x) \ll_r \frac{x}{a^2}(\log a).
\]

In the case $ w_1 > \frac{a}{2} $, let $ w_1 = (a+1)-z_1 $, where $ 1 \leq z_1 < \frac{a}{2}+1 $. We have
\[
	t_1 = \frac{(a+1)-w_1}{a^2+3a+3} > \frac{z_1}{a^2+3a+3} \geq \frac{z_1}{2a^2}
\]
and
\[
	t_3 \geq \frac{(a+2)w_1}{a^2+3a+3} > \frac{(a+2)a}{2(a^2+3a+3)} > \frac{1}{3}.
\]

If $ \Nm(\alpha) \leq x $, then by \Cref{lem:Estimates}, $ t_2^3 < x $ and $ a^4t_1t_2t_3 < x $, hence $ o_2^3 < x $ and $ a^2o_2z_1 < 6x $. Let
\[
	M_3(a, x) := \left\{(o_2, w) \in \Z_{\geq 1}^2 : o_2^3 < x, w_0 = 0, w_1 = (a+1)-z_1, 1 \leq z_1 < \frac{a}{2}+1, w_2 = 1, a^2o_2z_1 < 6x\right\}.
\]
The number of pairs can be estimated as
\[
	\#M_3(a, x) \leq \sum_{1 \leq o_2 < x^{1/3}}\frac{6x}{a^2o_2} \ll \frac{x}{a^2}\sum_{1 \leq o_2 < x^{1/3}}\frac{1}{o_2} \ll \frac{x}{a^2}(\log 2x) \ll_r \frac{x}{a^2}(\log a). 
\]

Finally,
\[
	\#\left\{\alpha \in R(a, S_{3,2}(a)) : \Nm(\alpha) \leq x\right\} \leq \#M_1(a, x)+\#M_2(a, x)+\#M_3(a, x) \ll_r \frac{x}{a^2}(\log a)^2.\qedhere
\]
\end{proof}

\begin{lemma}
\label{lem:S33}
Let $ a \in \Z $, $ a \geq 8 $. If $ 1 \leq x $, then
\[
	\#\left\{\alpha \in R(a, S_{3,3}(a)) : \Nm(\alpha)\leq x\right\} \ll \left(\frac{x}{a}\right)^{2/3}.
\]
\end{lemma}
\begin{proof}
Let $ \alpha \in R(a, S_{3,3}(a)) $, let $ t_1, t_2, t_3 \in \R_{>0} $ be such that $ \alpha = t_1+t_2\rho^2+t_3(\rho')^{-2} $, and let $ \tau(\alpha) = (m, -w, o) $, where $ (m, w, o) \in \Psi(a, S_{3,3}(a)) $. Let $ m = m_0(a, w)+1 $, $ o = o_0(a, w)+o_1 $, and $ o_2 = o_1-1 $. Since $ o_1 \geq 2 $, we have $ o_2 \geq 1 $. Since $ w_2 \geq 1 $, by \Cref{lem:S3t1t2},
\[
	t_1 = \frac{(a+1)w_2-w_1}{a^2+3a+3} = \frac{(a+1)w_2}{a^2+3a+3} \geq \frac{w_2}{2a}.
\]
By \Cref{lem:t1t2}, $ t_2 \geq o_1-1 = o_2 $. We also have
\[
	t_3 = \frac{w}{a^2+3a+3} = w_0+\frac{(a+2)w_1}{a^2+3a+3}+\frac{w_2}{a^2+3a+3} = \frac{w_2}{a^2+3a+3} \geq \frac{w_2}{2a^2}.
\]

If $ \Nm(\alpha) \leq x $, then by \Cref{lem:Estimates}, $ a^2t_1t_2^2 < x $ and $ a^4t_1t_2t_3 < x $, hence $ aw_2o_2^2 < 2x $ and $ ao_2w_2^2 < 4x $. Let
\[
	M(a, x) := \left\{(o_2, w_2) \in \Z_{\geq 1}^2 : aw_2o_2^2 < 2x, ao_2w_2^2 < 4x\right\}.
\]
The number of pairs in $ M(a, x) $ can be estimated as in \Cref{lem:S25}, and we get
\[
	\#M(a, x) \ll \left(\frac{x}{a}\right)^{2/3}.
\]
Thus,
\[
	\#\left\{\alpha \in R(a, S_{3,3}(a)) : \Nm(\alpha)\leq x\right\} \ll \left(\frac{x}{a}\right)^{2/3}.\qedhere
\]
\end{proof}

\begin{lemma}
\label{lemmaS34}
Let $ a \in \Z $, $ a \geq 8 $ and $ r > 0 $. If $ 1 \leq x \leq a^r $, then
\[
	\#\left\{\alpha \in Q(a, S_{3,4}(a)) : \Nm(\alpha)\leq x\right\} \ll_r \frac{x}{a^3}(\log a).
\]
The constant implied in the estimate depends only on $ r $.
\end{lemma}
\begin{proof}
Let $ \alpha \in R(a, S_{3,4}(a)) $, let $ t_1, t_2, t_3 \in \R_{>0} $ be such that $ \alpha = t_1+t_2\rho^2+t_3(\rho')^{-2} $, and let $ \tau(\alpha) = (m, -w, o) $, where $ (m, w, o) \in \Psi(a, S_{3,4}(a)) $. Let $ m = m_0(a, w)+1 $, $ o = o_0(a, w)+o_1 $, and $ o_2 = o_1-1 $. Since $ o_1 \geq 2 $, we have $ o_2 \geq 1 $. Since $ w_2 = 0 $, by \Cref{lem:S3t1t2},
\[
	t_1 = 1-\frac{w_1}{a^2+3a+3} \geq 1-\frac{a+1}{a^2+3a+3} \geq \frac{1}{2}.
\]
By \Cref{lem:t1t2}, $ t_2 \geq o_1-1 = o_2 $. We also have
\[
	t_3 = \frac{w}{a^2+3a+3} = w_0+\frac{(a+2)w_1}{a^2+3a+3}+\frac{w_2}{a^2+3a+3} \geq w_0.
\]

If $ \Nm(\alpha) \leq x $, then by \Cref{lem:Estimates}, $ t_2^3 < x $ and $ a^4t_1t_2t_3 < x $, hence $ o_2^3 < x $ and $ a^4o_2w_0 < 2x $. Let
\[
	M(a, x) := \left\{(o_2, w) \in \Z_{\geq 1}^2 : o_2^3 < x, w_0 \geq 1, w_1 \geq 0, w_2 = 0, a^4o_2w_0 < 2x\right\}.
\]

The number of pairs in $ M(a, x) $ can be estimated as
\[
	\#M_2(a, x) \leq \sum_{1 \leq o_2 < x^{1/3}}\sum_{1 \leq w_1 \leq a+1}\frac{2x}{a^4o_2} \ll \frac{x}{a^3}\sum_{1 \leq o_2 < x^{1/3}} \frac{1}{o_2} \ll \frac{x}{a^3}(\log 2x) \ll_r \frac{x}{a^3}(\log a).
\]

Thus,
\[
	\#\left\{\alpha \in R(a, S_{3,4}(a)) : \Nm(\alpha) \leq x\right\} \leq \#M(a, x) \ll_r \frac{x}{a^3}(\log a).\qedhere
\]
\end{proof}

\begin{lemma}
\label{lem:S35}
Let $ a \in \Z $, $ a \geq 8 $ and $ r > 0 $. If $ 1 \leq x \leq a^r $, then
\[
	\#\left\{\alpha \in R(a, S_{3,5}(a)) : \Nm(\alpha) \leq x\right\} \ll_r \frac{x}{a^3}(\log a).
\]
The constant implied in the estimate depends only on $ r $.
\end{lemma}
\begin{proof}
Let $ \alpha \in R(a, S_{3,5}(a)) $, let $ t_1, t_2, t_3 \in \R_{>0} $ be such that $ \alpha = t_1+t_2\rho^2+t_3(\rho')^{-2} $, and let $ \tau(\alpha) = (m, -w, o) $, where $ (m, w, o) \in \Psi(a, S_{3,5}(a)) $. Let $ m = m_0(a, w)+1 $, $ o = o_0(a, w)+o_1 $, and $ o_2 = o_1-1 $. Since $ o_1 \geq 2 $, we have $ o_2 \geq 1 $. Since $ w_2 = 0 $, by \Cref{lem:S3t1t2},
\[
	t_1 = 1-\frac{w_1}{a^2+3a+3} \geq 1-\frac{a+1}{a^2+3a+3} \geq \frac{1}{2}.
\]
By \Cref{lem:t1t2}, $ t_2 \geq o_1-1 = o_2 $. We also have
\[
	t_3 = \frac{w}{a^2+3a+3} = w_0+\frac{(a+2)w_1}{a^2+3a+3}+\frac{w_2}{a^2+3a+3} = \frac{(a+2)w_1}{a^2+3a+3} \geq \frac{w_1}{2a}.
\]

If $ \Nm(\alpha) \leq x $, then by \Cref{lem:Estimates}, $ t_2^3 < x $ and $ a^4t_1t_2t_3 < x $, hence $ o_2^3 < x $ and $ a^3o_2w_1 < 4x $. Let
\[
	M(a, x) := \left\{(o_2, w) \in \Z_{\geq 1}^2 : o_2^3 \leq x, w_0 = 0, w_1 \geq 1, w_2 = 0, a^3o_2w_1 < 4x\right\}.
\]

The number of pairs in $ M(a, x) $ can be estimated as
\[
	\#M(a, x) \leq \sum_{1 \leq o_2 < x^{1/3}}\frac{4x}{a^3o_2} \ll \frac{x}{a^3}(\log 2x) \ll_r \frac{x}{a^3}(\log a).
\]
Thus,
\[
	\#\left\{\alpha \in R(a, S_{3,5}(a)) : \Nm(\alpha) \leq x\right\} \leq \#M(a, x) \ll_r \frac{x}{a^3}(\log a).\qedhere
\]
\end{proof}

\begin{prop}
\label{prop:S3}
Let $ a \in \Z $, $ a \geq 8 $ and $ r > 0 $. If $ 1 \leq x \leq a^r $, then
\[
	\#\left\{\alpha \in R(a, S_3(a)) : \Nm(\alpha) \leq x\right\} \ll_r \frac{x}{a^2}(\log a)^2+\left(\frac{x}{a}\right)^{2/3}.
\]
The constant implied in the estimate depends only on $ r $.
\end{prop}
\begin{proof}
We have $ S_3(a) = S_{3,1}(a)\sqcup S_{3,2}(a)\sqcup \dots \sqcup S_{3,5}(a) $, hence $ R(a, S_3(a)) = R(a, S_{3,1}(a))\sqcup R(a, S_{3,2}(a))\sqcup \dots\sqcup R(a, S_{3,5}(a)) $. The proposition now follows from \Crefrange{lem:S31}{lem:S35}.
\end{proof}

\subsection{Estimate for \texorpdfstring{$ Q(a, S_4(a)) $}{Q(a,S4(a))}}
\label{sec:S4}

We recall that
\[
	S_4(a) = \left\{(m_1, o_1, w_0, w_1, w_2) \in \Z^2 \times W_1(a) : m_1 = 1, o_1 = 1\right\}.
\]

We write $ S_4(a) = S_{4,1}(a)\sqcup S_{4,2}(a)\sqcup \dots \sqcup S_{4,12}(a) $, where
\begin{align*}
	S_{4,1}(a)& := \left\{(m_1, o_1, w_0, w_1, w_2) \in S_4(a) : w_0 \geq 1, w_1 = 0, w_2 \geq 1\right\},\\
	S_{4,2}(a)& := \left\{(m_1, o_1, w_0, w_1, w_2) \in S_4(a) : w_0 = 0, w_1 = 0, w_2 \geq 1\right\},\\
	S_{4,3}(a)& := \left\{(m_1, o_1, w_0, w_1, w_2) \in S_4(a) : w_0 \geq 1, 1 \leq w_1 \leq a+1, w_2 = w_1+1\right\},\\
	S_{4,4}(a)& := \left\{(m_1, o_1, w_0, w_1, w_2) \in S_4(a) : w_0 = 0, 1 \leq w_1 \leq a+1, w_2 = w_1+1\right\},\\
	S_{4,5}(a)& := \left\{(m_1, o_1, w_0, w_1, w_2) \in S_4(a) : w_0 \geq 1, 1 \leq w_1 \leq a+1, w_2 \geq w_1+2\right\},\\
	S_{4,6}(a)& := \left\{(m_1, o_1, w_0, w_1, w_2) \in S_4(a) : w_0 = 0, 1 \leq w_1 \leq a+1, w_2 \geq w_1+2\right\},\\
	S_{4,7}(a)& := \left\{(m_1, o_1, w_0, w_1, w_2) \in S_4(a) : w_0 \geq 1, w_1 \geq w_2, w_2 \geq 1\right\},\\
	S_{4,8}(a)& := \left\{(m_1, o_1, w_0, w_1, w_2) \in S_4(a) : w_0 = 0, w_1 \geq w_2, w_2 \geq 1\right\},\\
	S_{4,9}(a)& := \left\{(m_1, o_1, w_0, w_1, w_2) \in S_4(a) : w_0 \geq 0, w_1 = a+2, w_2 = 0\right\},\\
	S_{4,10}(a)& := \left\{(m_1, o_1, w_0, w_1, w_2) \in S_4(a) : w_0 \geq 1, 1 \leq w_1 \leq a+1, w_2 = 0\right\},\\
	S_{4,11}(a)& := \left\{(m_1, o_1, w_0, w_1, w_2) \in S_4(a) : w_0 = 0, 1 \leq w_1 \leq a+1, w_2 = 0\right\},\\
	S_{4,12}(a)& := \left\{(m_1, o_1, w_0, w_1, w_2) \in S_4(a) : w_0 \geq 1, w_1 = 0, w_2 = 0\right\}.
\end{align*}

Now we again parametrize $ w $ as $ w = (a^2+3a+3)w_0+(a+1)w_1+w_2 $, where $ w_1 \leq a+2 $, $ w_2 \leq a $, and $ w_2 = 0 $ if $ w_1 = a+2 $, so that $ \phi_a(w) = (w_0,w_1,w_2) $. As before, we write $ m = m_0(a,w)+m_1 $ and $ o = o_0(a,w)+o_1 $, where
\begin{align*}
	m_0(a,w)& = \left\lfloor-\frac{a+1}{a^2+3a+3}w\right\rfloor,\\
	o_0(a,w)& = \left\lfloor\frac{a+2}{a^2+3a+3}w\right\rfloor.
\end{align*}

This time, we assume $ m_1 = 1 $ and $ o_1 = 1 $, hence
\begin{align*}
	t_1& = m+\frac{a+1}{a^2+3a+3}w = \left\lfloor-\frac{a+1}{a^2+3a+3}w\right\rfloor+1+\frac{a+1}{a^2+3a+3}w.\\
	t_2& = o-\frac{a+2}{a^2+3a+3}w = \left\lfloor\frac{a+2}{a^2+3a+3}w\right\rfloor+1-\frac{a+2}{a^2+3a+3}w.
\end{align*}

\begin{lemma}
\label{lem:S4t1t2}
Let $ a \in \Z_{\geq 1} $, $ w \in \Z_{\geq 1} $, $ (w_0,w_1,w_2) = \phi_a(w) $, $ m = m_0(a,w)+1 $, $ o = o_0(a,w)+1 $. If $ \alpha = t_1+t_2\rho^2+t_3(\rho')^{-2} $ is such that $ \tau(\alpha) = (m,-w,o) $, then
\begin{align*}
	t_1& = \frac{(a+1)w_2-(a+2)w_1}{a^2+3a+3}+\begin{cases}
		2,&\text{if }w_1 = a+2,\\
		1,&\text{if }w_1 \leq a+1\text{ and }w_1 \geq w_2,\\
		0,&\text{if }w_1 \leq a+1\text{ and }w_1 < w_2,
	\end{cases}\\
	t_2& = \begin{cases}
		1,&\text{if }w_1 = w_2 = 0,\\
		1+\frac{w_1-(a+2)w_2}{a^2+3a+3},&\text{if }w_2 \geq 1,\\
		\frac{w_1}{a^2+3a+3},&\text{if }w_1 \geq 1, w_2 = 0.
	\end{cases}
\end{align*}
\end{lemma}
\begin{proof}
We have
\[
	t_1 = 1-\left(-\frac{a+1}{a^2+3a+3}w-\left\lfloor-\frac{a+1}{a^2+3a+3}w\right\rfloor\right).
\]
By \Cref{lem:W1},
\[
	\frac{a+1}{a^2+3a+3}w-\left\lfloor-\frac{a+1}{a^2+3a+3}w\right\rfloor = \frac{(a+2)w_1-(a+1)w_2}{a^2+3a+3}+\begin{cases}
		-1,&\text{if }w_1 = a+2,\\
		0,&\text{if }w_1 \leq a+1\text{ and }w_1 \geq w_2,\\
		1,&\text{if }w_1 \leq a+1\text{ and }w_1 < w_2.
	\end{cases}
\]
The formula for $ t_1 $ follows. We also have
\[
	t_2 = 1-\left(\frac{a+2}{a^2+3a+3}w-\left\lfloor\frac{a+2}{a^2+3a+3}w\right\rfloor\right)
\]
and by \Cref{lem:W2},
\[
	\frac{a+2}{a^2+3a+3}w-\left\lfloor\frac{a+2}{a^2+3a+3}w\right\rfloor = \begin{cases}
		0,&\text{if }w_1 = w_2 = 0,\\
		\frac{(a+2)w_2-w_1}{a^2+3a+3},&\text{if }w_2 \geq 1,\\
		1-\frac{w_1}{a^2+3a+3},&\text{if }w_1 \geq 1, w_2 = 0.
	\end{cases}
\]
The formula for $ t_2 $ follows.
\end{proof}

\begin{lemma}
\label{lem:S41}
Let $ a \in \Z $, $ a \geq 8 $. If $ x \geq 1 $, then
\[
	\#\left\{\alpha \in Q(a, S_{4,1}(a)) : \Nm(\alpha)\leq x\right\} \ll \frac{x}{a^3}(\log a).
\]
\end{lemma}
\begin{proof}
Let $ \alpha \in Q(a, S_{4,1}(a)) $, let $ t_1, t_2, t_3 \in \R_{>0} $ be such that $ \alpha = t_1+t_2\rho^2+t_3(\rho')^{-2} $, and let $ \tau(\alpha) = (m, -w, o) $, where $ (m, w, o) \in \Phi(a, S_{4,1}(a)) $. Since $ w_1 = 0 $ and $ w_2 \geq 1 $, by \Cref{lem:S4t1t2},
\begin{align*}
	t_1& = \frac{(a+1)w_2-(a+2)w_1}{a^2+3a+3} = \frac{(a+1)w_2}{a^2+3a+3} \geq \frac{w_2}{2a},\\
	t_2& = 1+\frac{w_1-(a+2)w_2}{a^2+3a+3} = 1-\frac{(a+2)w_2}{a^2+3a+3}.
\end{align*}
We also have
\[
	t_3 = \frac{w}{a^2+3a+3} = w_0+\frac{(a+1)w_1}{a^2+3a+3}+\frac{w_2}{a^2+3a+3} \geq w_0.
\]

We consider two cases. First, assume $ w_2 \leq \frac{a}{2} $. Then
\[
	t_2 = 1-\frac{(a+2)w_2}{a^2+3a+3} \geq 1-\frac{(a+2)a}{2(a^2+3a+3)} > \frac{1}{2}.
\]

If $ \Nm(\alpha) \leq x $, then by \Cref{lem:Estimates}, $ a^4t_1t_2t_3 < x $, hence $ a^3w_0w_2 < 4x $. Let
\[
	M_1(a, x) := \left\{(w_0, w_1, w_2) \in \Z_{\geq 0}^3 : w_0 \geq 1, w_1 = 0, 1 \leq w_2 \leq \frac{a}{2}, a^3w_0w_2 < 4x\right\}.
\]
The number of triples in $ M_1(a, x) $ can be estimated as
\[
	\#M_1(a, x) \leq \sum_{1 \leq w_2 \leq \frac{a}{2}}\frac{4x}{a^3w_2} \ll \frac{x}{a^3}\sum_{1 \leq w_2 \leq \frac{a}{2}}\frac{1}{w_2} \ll \frac{x}{a^3}(\log a).
\]

Secondly, assume $ w_2 > \frac{a}{2} $. Let $ w_2 = (a+1)-z_2 $, where $ 1 \leq z_2 < \frac{a}{2}+1 $. We have
\begin{align*}
	t_1& \geq \frac{w_2}{2a} > \frac{1}{4},\\
	t_2& = 1-\frac{(a+2)((a+1)-z_2)}{a^2+3a+3} = 1-\frac{(a+2)(a+1)}{a^2+3a+3}+\frac{(a+2)z_2}{a^2+3a+3} \geq \frac{z_2}{2a}.
\end{align*}

If $ \Nm(\alpha) \leq x $, then by \Cref{lem:Estimates}, $ a^4t_1t_2t_3 < x $, hence $ a^3w_0z_2 < 8x $. Let
\[
	M_2(a, x) := \left\{(w_0, w_1, w_2) \in \Z_{\geq 0}^3 : w_0 \geq 1, w_1 = 0, w_2 = (a+1)-z_2, 1 \leq z_2 < \frac{a}{2}+1, a^3w_0z_2 \leq 8x\right\}.
\]
The number of triples in $ M_2(a, x) $ can be estimated as
\[
	\#M_2(a, x) \leq \sum_{1 \leq z_2 < \frac{a}{2}+1}\frac{8x}{a^3z_2} \ll \frac{x}{a^3}\sum_{1 \leq z_2 < \frac{a}{2}+1}\frac{1}{z_2} \ll \frac{x}{a^3}(\log a).
\]

Finally,
\[
	\#\left\{\alpha \in Q(a, S_{4,1}(a)) : \Nm(\alpha) \leq x\right\} \leq \#M_1(a, x)+\#M_2(a, x) \ll \frac{x}{a^3}(\log a).\qedhere
\]
\end{proof}

\begin{lemma}
\label{lem:S42}
Let $ a \in \Z $, $ a \geq 8 $. If $ x \geq 1 $, then
\[
	\#\left\{\alpha \in Q(a, S_{4,2}(a)) : \Nm(\alpha)\leq x\right\} \ll \left(\frac{x}{a}\right)^{1/2}+\frac{x}{a^2}.
\]
\end{lemma}
\begin{proof}
Let $ \alpha \in Q(a, S_{4,2}(a)) $, let $ t_1, t_2, t_3 \in \R_{>0} $ be such that $ \alpha = t_1+t_2\rho^2+t_3(\rho')^{-2} $, and let $ \tau(\alpha) = (m, -w, o) $, where $ (m, w, o) \in \Phi(a, S_{4,2}(a)) $. Since $ w_1 = 0 $ and $ w_2 \geq 1 $, by \Cref{lem:S4t1t2},
\begin{align*}
	t_1& = \frac{(a+1)w_2-(a+2)w_1}{a^2+3a+3} = \frac{(a+1)w_2}{a^2+3a+3} \geq \frac{w_2}{2a},\\
	t_2& = 1+\frac{w_1-(a+2)w_2}{a^2+3a+3} = 1-\frac{(a+2)w_2}{a^2+3a+3}.
\end{align*}
We also have
\[
	t_3 = \frac{w}{a^2+3a+3} = w_0+\frac{(a+1)w_1}{a^2+3a+3}+\frac{w_2}{a^2+3a+3} = \frac{w_2}{a^2+3a+3} \geq \frac{w_2}{2a^2}.
\]

We consider two cases. First, assume $ w_2 \leq \frac{a}{2} $. Then
\[
	t_2 = 1-\frac{(a+2)w_2}{a^2+3a+3} \geq 1-\frac{(a+2)a}{2(a^2+3a+3)} > \frac{1}{2}.
\]

If $ \Nm(\alpha) \leq x $, then by \Cref{lem:Estimates}, $ a^4t_1t_2t_3 < x $, hence $ aw_2^2 < 8x $. Let
\[
	M_1(a, x) := \left\{(w_0, w_1, w_2) \in \Z_{\geq 0}^3 : w_0 = 0, w_1 = 0, 1 \leq w_2 \leq \frac{a}{2}, aw_2^2 < 8x\right\}.
\]
We have $ \#M_1(a, x) \ll \left(\frac{x}{a}\right)^{1/2} $.

Secondly, assume $ w_2 > \frac{a}{2} $ and let $ w_2 = (a+1)-z_2 $, where $ 1 \leq z_2 < \frac{a}{2}+1 $. We have
\begin{align*}
	t_1& \geq \frac{w_2}{2a} > \frac{1}{4},\\
	t_2& = 1-\frac{(a+2)((a+1)-z_2)}{a^2+3a+3} = 1-\frac{(a+2)(a+1)}{a^2+3a+3}+\frac{(a+2)z_2}{a^2+3a+3} \geq \frac{z_2}{2a},\\
	t_3& \geq \frac{w_2}{2a^2} > \frac{1}{4a}.
\end{align*}

If $ \Nm(\alpha) \leq x $, then by \Cref{lem:Estimates}, $ a^4t_1t_2t_3 < x $, hence $ a^2z_2 < 32 x $. Let
\[
	M_2(a, x) := \left\{(w_0, w_1, w_2) \in \Z_{\geq 0}^3 : w_0 = 0, w_1 = 0, w_2 = (a+1)-z_2, 1 \leq z_2 < \frac{a}{2}+1, a^2z_2 < 32x\right\}.
\]
We have $ \#M_2(a, x) \ll \frac{x}{a^2} $.

Finally,
\[
	\#\left\{\alpha \in Q(a, S_{4,2}(a)) : \Nm(\alpha) \leq x\right\} \leq \#M_1(a, x)+\#M_2(a, x) \ll \left(\frac{x}{a}\right)^{1/2}+\frac{x}{a^2}.\qedhere
\]
\end{proof}

\begin{lemma}
\label{lem:S43}
Let $ a \in \Z $, $ a \geq 8 $. If $ x \geq 1 $, then
\[
	\#\left\{\alpha \in Q(a, S_{4,3}(a)) : \Nm(\alpha) \leq x\right\} \ll \frac{x}{a^2}+\left(\frac{x}{a}\right)^{2/3}.
\]
\end{lemma}
\begin{proof}
Let $ \alpha \in Q(a, S_{4,3}(a)) $, let $ t_1, t_2, t_3 \in \R_{>0} $ be such that $ \alpha = t_1+t_2\rho^2+t_3(\rho')^{-2} $, and let $ \tau(\alpha) = (m,-w,o) $, where $ (m,w,o) \in \Phi(a,S_{4,3}(a)) $. Since $ w_2 \leq a $ and $ w_2 = w_1+1 $, we have $ w_1 \leq a-1 $.

By \Cref{lem:S4t1t2},
\begin{align*}
	t_1& = \frac{(a+1)w_2-(a+2)w_1}{a^2+3a+3} = \frac{(a+1)-w_1}{a^2+3a+3},\\
	t_2& = 1+\frac{w_1-(a+2)w_2}{a^2+3a+3} = 1+\frac{w_1-(a+2)(w_1+1)}{a^2+3a+3} = 1-\frac{a+2}{a^2+3a+3}-\frac{(a+1)w_1}{a^3+3a+3}.
\end{align*}
We also have
\[
	t_3 = \frac{w}{a^2+3a+3} = w_0+\frac{(a+1)w_1}{a^2+3a+3}+\frac{w_2}{a^2+3a+3} \geq w_0.
\]

We consider two cases. First, assume $ w_1 \leq \frac{a}{2} $. We get
\begin{align*}
	t_1& \geq \frac{(a+1)-\frac{a}{2}}{a^2+3a+3} = \frac{a+2}{2(a^2+3a+3)} \geq \frac{1}{4a},\\
	t_2& \geq 1-\frac{a+2}{a^2+3a+3}-\frac{(a+1)a}{2(a^2+3a+3)} \geq \frac{1}{3}.
\end{align*}
By \Cref{lem:Estimates}, $ a^4t_1t_2t_3 < x $, hence $ a^3w_0 < 12x $. If we let
\[
	M_1(a,x) := \left\{(w_0, w_1, w_2) \in \Z_{\geq 0}^3 : w_0 \geq 1, 1 \leq w_1 \leq \frac{a}{2}, w_2 = w_1+1, a^3w_0 < 12x\right\},
\]
then
\[
	\#M_1(a,x) \ll \frac{x}{a^2}.
\]

Secondly, assume $ w_1 > \frac{a}{2} $ and let $ w_1 = a-z_1 $, where $ 1 \leq z_1 < \frac{a}{2} $. We get
\begin{align*}
	t_1& = \frac{1+z_1}{a^2+3a+3} \geq \frac{z_1}{2a^2},\\
	t_2& = 1-\frac{a+2}{a^2+3a+3}-\frac{(a+1)a}{a^2+3a+3}+\frac{(a+1)z_1}{a^2+3a+3} \geq \frac{z_1}{2a}.
\end{align*}
By \Cref{lem:Estimates}, $ a^2t_2t_3^2 < 2x $ and $ a^4t_1t_2t_3 < x $, hence $ az_1w_0^2 < 4x $ and $ az_1^2w_0 < 4x $. If we let
\[
	M_2(a,x) := \left\{(w_0, w_1, w_2) \in \Z_{\geq 0}^3 : w_0 \geq 1, w_1 = a-z_1, 1 \leq z_1 < \frac{a}{2}, w_2 = w_1+1, az_1w_0^2 < 4x, az_1^2w_0<4x\right\},
\]
then we can estimate as in the proof of \Cref{lem:S25}
\[
	\#M_2(a,x) \ll \left(\frac{x}{a}\right)^{2/3}.
\]

Finally,
\[
	\#\left\{\alpha \in Q(a, S_{4,3}(a)) : \Nm(\alpha) \leq x\right\} \leq \#M_1(a,x)+\#M_2(a,x) \ll \frac{x}{a^2}+\left(\frac{x}{a}\right)^{2/3}.\qedhere
\]
\end{proof}

\begin{lemma}
\label{lem:S44}
Let $ a \in \Z $, $ a \geq 8 $. If $ x \geq 1 $, then
\[
	\#\left\{\alpha \in Q(a, S_{4,4}(a)) : \Nm(\alpha) \leq x\right\} \ll \frac{x}{a^2}+\left(\frac{x}{a}\right)^{1/2}.
\]
\end{lemma}
\begin{proof}
Let $ \alpha \in Q(a,S_{4,4}(a)) $, let $ t_1, t_2, t_3 \in \R_{>0} $ be such that $ \alpha = t_1+t_2\rho^2+t_3(\rho')^{-2} $, and let $ \tau(\alpha) = (m,-w,o) $, where $ (m, w, o) \in \Phi(a, S_{4,4}(a)) $. We have
\[
	t_3 = \frac{w}{a^2+3a+3} = w_0+\frac{(a+1)w_1}{a^2+3a+3}+\frac{w_2}{a^2+3a+3} \geq \frac{(a+1)w_1}{a^2+3a+3} > \frac{w_1}{2a}.
\]

We consider two cases. First, assume $ w_1 \leq \frac{a}{2} $. As in the proof of \Cref{lem:S43}, $ t_1 \geq \frac{1}{4a} $ and $ t_2 \geq \frac{1}{3} $. If $ \Nm(\alpha) \leq x $, then $ a^4t_1t_2t_3 < x $ by \Cref{lem:Estimates}, hence $ a^2w_1 \leq 24x $. If we let
\[
	M_1(a,x) := \left\{(w_0,w_1,w_2) \in \Z_{\geq 0}^3 : w_0 = 0, 1 \leq w_1 \leq \frac{a}{2}, w_2 = w_1+1, a^2w_1 \leq 24x\right\},
\]
then
\[
	\#M_1(a,x) \ll \frac{x}{a^2}.
\]

Secondly, assume $ w_1 > \frac{a}{2} $ and let $ w_1 = a-z_1 $, where $ 1 \leq z_1 < \frac{a}{2} $. As in the proof of \Cref{lem:S43}, $ t_1 \geq \frac{z_1}{2a^2} $ and $ t_2 \geq \frac{z_1}{2a} $. Moreover, $ t_3 > \frac{1}{4} $. If $ \Nm(\alpha) \leq x $, then $ a^4t_1t_2t_3 < x $, hence $ az_1^2 < 16x $. Letting
\[
	M_2(a,x) := \left\{(w_0,w_1,w_2) \in \Z_{\geq 0}^3 : w_0 = 0, w_1 = a-z_1, 1 \leq z_1 < \frac{a}{2}, w_2 = w_1+1, az_1^2 < 16x\right\},
\]
we have
\[
	\#M_2(a,x) \ll \left(\frac{x}{a}\right)^{1/2}.
\]

Finally,
\[
	\#\left\{\alpha \in Q(a, S_{4,4}(a)) : \Nm(\alpha) \leq x\right\} \leq \#M_1(a,x)+\#M_2(a,x) \ll \frac{x}{a^2}+\left(\frac{x}{a}\right)^{1/2}.\qedhere
\]
\end{proof}

\begin{lemma}
\label{lem:S45}
Let $ a \in \Z $, $ a \geq 8 $. If $ x \geq 1 $, then
\[
	\#\left\{\alpha \in Q(a, S_{4,5}(a)) : \Nm(\alpha) \leq x\right\} \ll \frac{x}{a^2}(\log a)^2.
\]
\end{lemma}
\begin{proof}
Let $ \alpha \in Q(a, S_{4,5}(a)) $, let $ t_1, t_2, t_3 \in \R_{>0} $ be such that $ \alpha = t_1+t_2\rho^2+t_3(\rho')^{-2} $, and let $ \tau(\alpha) = (m,-w,o) $, where $ (m,w,o) \in \Phi(a,S_{4,5}(a)) $. Let $ z_1 = w_2-w_1-1 $ and $ z_2 = (a+1)-w_2 $. Since $ w_2 \geq w_1+2 $, $ w_1 \geq 1 $, and $ w_2 \leq a $, we have $ 1 \leq z_1 \leq a-2 $ and $ 1 \leq z_2 \leq a-2 $. By \Cref{lem:S4t1t2},
\begin{align*}
	t_1& = \frac{(a+1)w_2-(a+2)w_1}{a^2+3a+3} = \frac{(a+1)w_2-(a+2)(w_2-1-z_1)}{a^2+3a+3} \geq \frac{(a+2)-w_2+(a+2)z_1}{a^2+3a+3} \geq \frac{z_1}{2a},\\
	t_2& = 1+\frac{w_1-(a+2)w_2}{a^2+3a+3} = 1+\frac{w_1-(a+2)(a+1)}{a^2+3a+3}+\frac{(a+2)z_2}{a^2+3a+3} \geq \frac{z_2}{2a}.
\end{align*}
We also have
\[
	t_3 = \frac{w}{a^2+3a+3} = w_0+\frac{(a+1)w_1}{a^2+3a+3}+\frac{w_2}{a^2+3a+3} \geq w_0.
\]

By \Cref{lem:Estimates}, $ a^4t_1t_2t_3 < x $, hence $ a^2z_1z_2w_0 < 4x $. If we let
\[
	M(a,x) := \left\{(w_0,z_1,z_2) \in \Z_{\geq 0}^3 : w_0 \geq 1, 1 \leq z_1 \leq a-2, 1 \leq z_2 \leq a-2,\\
	 a^2z_1z_2w_0 < 4x\right\},
\]
then
\[
	\#M(a,x) \leq \sum_{1 \leq z_1 \leq a-2}\sum_{1 \leq z_2 \leq a-2}\frac{4x}{a^2z_1z_2} \ll \frac{x}{a^2}\sum_{1 \leq z_1 \leq a-2}\sum_{1 \leq z_2 \leq a-2}\frac{1}{z_1z_2} \ll \frac{x}{a^2}(\log a)^2.
\]

Finally,
\[
	\#\left\{\alpha \in Q(a,S_{4,5}(a)) : \Nm(\alpha)\leq x\right\} \leq \#M(a,x) \ll \frac{x}{a^2}(\log a)^2.\qedhere
\]
\end{proof}

\begin{lemma}
\label{lem:S46}
Let $ a \in \Z $, $ a \geq 8 $. If $ x \geq 1 $, then
\[
	\#\left\{\alpha \in Q(a,S_{4,6}(a)) : \Nm(\alpha)\leq x\right\} \ll \frac{x}{a^2}(\log a).
\]
\end{lemma}
\begin{proof}
Let $ \alpha \in Q(a, S_{4,6}(a)) $, let $ t_1, t_2, t_3 \in \R_{>0} $ be such that $ \alpha = t_1+t_2\rho^2+t_3(\rho')^{-2} $, and let $ \tau(\alpha) = (m,-w,o) $, where $ (m,w,o) \in \Phi(a,S_{4,6}(a)) $. We again let $ z_1 = w_2-w_1-1 $ and $ z_2 = (a+1)-w_2 $. As in \Cref{lem:S45}, $ 1 \leq z_1 \leq a-2 $, $ 1 \leq z_2 \leq a-2 $, and $ t_1 \geq \frac{z_1}{2a} $, $ t_2 \geq \frac{z_2}{2a} $. We also have
\[
	t_3 = \frac{w}{a^2+3a+3} = w_0+\frac{(a+1)w_1}{a^2+3a+3}+\frac{w_2}{a^2+3a+3} \geq \frac{(a+1)w_1}{a^2+3a+3} > \frac{w_1}{2a}.
\]

By \Cref{lem:Estimates}, $ a^4t_1t_2t_3 < x $. Now we distinguish four cases. If $ z_1 \leq \frac{a}{3} $ and $ z_2 \leq \frac{a}{3} $, then $ w_1 = a-z_1-z_2 \geq \frac{a}{3} $, hence $ t_3 > \frac{1}{6} $, and it follows that $ a^2z_1z_2 < 24x $. If we let
\[
	M_1(a,x) := \left\{(w_0, z_1,z_2) \in \Z_{\geq 0}^3 : w_0 = 0, 1 \leq z_1 \leq \frac{a}{3}, 1 \leq z_2 \leq \frac{a}{3}, a^2z_1z_2 < 24x\right\},
\]
then
\[
	\#M_1(a,x) \leq \sum_{1 \leq z_1 \leq \frac{a}{3}}\frac{24x}{a^2z_1} \ll \frac{x}{a^2}\sum_{1 \leq z_1 \leq \frac{a}{3}}\frac{1}{z_1} \ll \frac{x}{a^2}(\log a).
\]
If $ z_1 \leq \frac{a}{3} $ and $ z_2 > \frac{a}{3} $, then $ t_2 \geq \frac{z_2}{2a} > \frac{1}{6} $, and it follows that $ a^2z_1w_1 < 24x $. We note that $ z_2 $ is determined in terms of $ z_1 $ and $ w_1 $ because $ z_2 = (a+1)-w_2 = a-z_1-w_1 $. If we let
\[
	M_2(a,x) := \left\{(w_0, z_1, z_2) \in \Z_{\geq 0}^3 : w_0 = 0, 1 \leq z_1 \leq \frac{a}{3}, z_2 = a-z_1-w_1, a^2z_1w_1 < 24x\right\},
\]
then
\[
	\#M_2(a,x) \ll \frac{x}{a^2}(\log a).
\]
If $ z_1 > \frac{a}{3} $ and $ z_2 \leq \frac{a}{3} $, then $ t_1 \geq \frac{z_1}{2a} > \frac{1}{6} $ and $ a^2z_2w_1 < 24x $. We note that $ z_1 $ is determined in terms of $ z_2 $ and $ w_1 $ because $ z_1 = w_2-w_1-1 = a-z_2-w_1 $. If we let
\[
	M_3(a,x) := \left\{(w_0, z_1, z_2) \in \Z_{\geq 0}^3 : w_0 = 0, 1 \leq z_2 \leq \frac{a}{3}, z_1 = a-z_2-w_1, a^2z_2w_1 < 24x\right\},
\]
then
\[
	\#M_3(a,x) \ll \frac{x}{a^2}(\log a)
\]
as in the previous case. If $ z_1 > \frac{a}{3} $ and $ z_2 > \frac{a}{3} $, then $ t_1 > \frac{1}{6} $ and $ t_2 > \frac{1}{6} $, hence $ a^3w_1 < 72x $. If we let
\[
	M_4(a,x) := \left\{(w_0,z_1,z_2) \in \Z_{\geq 0}^3 : w_0 = 0, \frac{a}{3} < z_1 \leq a-2, z_2 = a-z_1-w_1, a^3w_1 < 72x\right\},
\]
then
\[
	\#M_4(a,x) \ll \frac{x}{a^2}.
\]
Finally,
\[
	\#\left\{\alpha \in Q(a,S_{4,6}(a)) : \Nm(\alpha)\leq x\right\} \leq \#M_1(a,x)+\#M_2(a,x)+\#M_3(a,x)+\#M_4(a,x) \ll \frac{x}{a^2}(\log a).\qedhere
\]
\end{proof}

\begin{lemma}
\label{lem:S47}
Let $ a \in \Z $, $ a \geq 8 $. If $ x \geq 1 $, then
\[
	\#\left\{\alpha \in Q(a, S_{4,7}(a)) : \Nm(\alpha) \leq x\right\} \ll \frac{x}{a^2}(\log a).
\]
\end{lemma}
\begin{proof}
Let $ \alpha \in Q(a, S_{4,7}(a)) $, let $ t_1, t_2, t_3 \in \R_{>0} $ be such that $ \alpha = t_1+t_2\rho^2+t_3(\rho')^{-2} $, and let $ \tau(\alpha) = (m, -w, o) $, where $ (m, w, o) \in \Phi(a, S_{4,7}(a)) $. Since $ w_1 \geq w_2 $ and $ w_2 \geq 1 $, by \Cref{lem:S4t1t2},
\begin{align*}
	t_1& = \frac{(a+1)w_2-(a+2)w_1}{a^2+3a+3}+1 \geq \frac{(a+1)w_2}{a^2+3a+3}+1-\frac{(a+2)(a+1)}{a^2+3a+3} \geq \frac{(a+1)w_2}{a^2+3a+3} \geq \frac{w_2}{2a},\\
	t_2& = 1+\frac{w_1-(a+2)w_2}{a^2+3a+3} \geq 1-\frac{(a+2)w_2}{a^2+3a+3}.
\end{align*}
We also have
\[
	t_3 = \frac{w}{a^2+3a+3} = w_0+\frac{(a+1)w_1}{a^2+3a+3}+\frac{w_2}{a^2+3a+3} \geq w_0.
\]

We consider two cases. First, assume $ w_2 \leq \frac{a}{2} $. Then
\[
	t_2 \geq 1-\frac{(a+2)w_2}{a^2+3a+3} \geq 1-\frac{(a+2)a}{2(a^2+3a+3)} \geq \frac{1}{2}.
\]

If $ \Nm(\alpha) \leq x $, then by \Cref{lem:Estimates}, $ a^4t_1t_2t_3 < x $, hence $ a^3w_0w_2 < 4x $. If we let
\[
	M_1(a, x) := \left\{(w_0, w_1, w_2) \in \Z_{\geq 0}^3 : w_0 \geq 1, w_1 \geq w_2, 1 \leq w_2 \leq \frac{a}{2}, a^3w_0w_2 < 4x\right\},
\]
then
\[
	\#M_1(a, x) \leq \sum_{1 \leq w_2 \leq \frac{a}{2}}\sum_{w_2 \leq w_1 \leq a+1}\frac{4x}{a^3w_2} \ll \frac{x}{a^2}\sum_{1 \leq w_2 \leq \frac{a}{2}}\frac{1}{w_2} \ll \frac{x}{a^2}(\log a).
\]

Secondly, assume $ w_2 > \frac{a}{2} $ and let $ w_2 = (a+1)-z_2 $, where $ 1 \leq z_2 < \frac{a}{2}+1 $. In this case we have
\begin{align*}
	t_1& \geq \frac{w_2}{2a} > \frac{1}{4},\\
	t_2& \geq 1-\frac{(a+2)w_2}{a^2+3a+3} = 1-\frac{(a+2)((a+1)-z_2)}{a^2+3a+3} \geq \frac{(a+2)z_2}{a^2+3a+3} \geq \frac{z_2}{2a}.
\end{align*}

If $ \Nm(\alpha) \leq x $, then by \Cref{lem:Estimates}, $ a^4t_1t_2t_3 < x $, hence $ a^3w_0z_2 < 8x $. If we let
\[
	M_2(a, x) := \left\{(w_0, w_1, w_2) \in \Z_{\geq 0}^3 : w_0 \geq 1, w_1 \geq w_2, w_2 = (a+1)-z_2, 1 \leq z_2 < \frac{a}{2}+1, a^3w_0z_2 < 8x\right\},
\]
then
\[
	\#M_2(a, x) \leq \sum_{1 \leq z_2 < \frac{a}{2}+1}\sum_{1 \leq w_1 \leq a+1}\frac{8x}{a^3z_2} \ll \frac{x}{a^2}\sum_{1 \leq z_2 < \frac{a}{2}+1}\frac{1}{z_2} \ll \frac{x}{a^2}(\log a).
\]

Finally,
\[
	\#\left\{\alpha \in Q(a, S_{4,7}(a)) : \Nm(\alpha) \leq x\right\} \leq \#M_1(a, x)+\#M_2(a, x) \ll \frac{x}{a^2}(\log a).\qedhere
\]
\end{proof}

\begin{lemma}
\label{lem:S48}
Let $ a \in \Z $, $ a \geq 8 $. If $ x \geq 1 $, then
\[
	\#\left\{\alpha \in Q(a, S_{4,8}(a)) : \Nm(\alpha) \leq x\right\} \ll \frac{x}{a^2}(\log a).
\]
\end{lemma}
\begin{proof}
Let $ \alpha \in Q(a, S_{4,8}(a)) $, let $ t_1, t_2, t_3 \in \R_{>0} $ be such that $ \alpha = t_1+t_2\rho^2+t_3(\rho')^{-2} $, and let $ \tau(\alpha) = (m, -w, o) $, where $ (m, w, o) \in \Phi(a, S_{4,8}(a)) $. Since $ w_1 \geq w_2 $ and $ w_2 \geq 1 $, we get as in the proof of \Cref{lem:S47}
\begin{align*}
	t_1& \geq \frac{w_2}{2a},\\
	t_2& \geq 1-\frac{(a+2)w_2}{a^2+3a+3}.
\end{align*}
We also have
\[
	t_3 = \frac{w}{a^2+3a+3} = w_0+\frac{(a+1)w_1}{a^2+3a+3}+\frac{w_2}{a^2+3a+3} \geq \frac{(a+1)w_1}{a^2+3a+3} \geq \frac{w_1}{2a}.
\]

We consider two cases. First, assume $ w_2 \leq \frac{a}{2} $. Then
\[
	t_2 \geq 1-\frac{(a+2)w_2}{a^2+3a+3} \geq 1-\frac{(a+2)a}{2(a^2+3a+3)} \geq \frac{1}{2}.
\]

If $ \Nm(\alpha) \leq x $, then by \Cref{lem:Estimates}, $ a^4t_1t_2t_3 < x $, hence $ a^2w_1w_2 < 8x $. If we let
\[
	M_1(a, x) := \left\{(w_0, w_1, w_2) \in \Z_{\geq 0}^3 : w_0 = 0, w_1 \geq w_2, 1 \leq w_2 \leq \frac{a}{2}, a^2w_1w_2 < 8x\right\},
\]
then
\[
	\#M_1(a, x) \leq \sum_{1 \leq w_2 \leq \frac{a}{2}}\frac{8x}{a^2w_2} \ll \frac{x}{a^2}\sum_{1 \leq w_2 \leq \frac{a}{2}}\frac{1}{w_2} \ll \frac{x}{a^2}(\log a).
\]

Secondly, assume $ w_2 > \frac{a}{2} $ and let $ w_2 = (a+1)-z_2 $, where $ 1 \leq z_2 < \frac{a}{2}+1 $. Then
\begin{align*}
	t_1& \geq \frac{w_2}{2a} > \frac{1}{4},\\
	t_2& \geq 1-\frac{(a+2)w_2}{a^2+3a+3} \geq 1-\frac{(a+2)((a+1)-z_2)}{a^2+3a+3} \geq \frac{(a+2)z_2}{a^2+3a+3} \geq \frac{z_2}{2a}.
\end{align*}

If $ \Nm(\alpha) \leq x $, then by \Cref{lem:Estimates}, $ a^4t_1t_2t_3 < x $, hence $ a^2w_1z_2 < 16x $. If we let
\[
	M_2(a, x) := \left\{(w_0, w_1, w_2) \in \Z_{\geq 0}^3 : w_0 = 0, w_1 \geq w_2, w_2 = (a+1)-z_2, 1 \leq z_2 < \frac{a}{2}+1, a^2w_1z_2 < 16x\right\},
\]
then
\[
	\#M_2(a, x) \leq \sum_{1 \leq z_2 < \frac{a}{2}+1}\frac{16x}{a^2z_2} \ll \frac{x}{a^2}\sum_{1 \leq z_2 < \frac{a}{2}+1}\frac{1}{z_2} \ll \frac{x}{a^2}(\log a).
\]

Finally,
\[
	\#\left\{\alpha \in Q(a, S_{4,8}(a)) : \Nm(\alpha) \leq x\right\} \leq \#M_1(a, x)+\#M_2(a, x) \ll \frac{x}{a^2}(\log a).\qedhere
\]
\end{proof}

\begin{lemma}
\label{lem:S49}
Let $ a \in \Z $, $ a \geq 8 $. If $ x \geq 1 $, then
\[
	\#\left\{\alpha \in Q(a, S_{4,9}(a)) : \Nm(\alpha) \leq x\right\} \ll \frac{x}{a^3}.
\]
\end{lemma}
\begin{proof}
Let $ \alpha \in Q(a, S_{4,9}(a)) $, let $ t_1, t_2, t_3 \in \R_{>0} $ be such that $ \alpha = t_1+t_2\rho^2+t_3(\rho')^{-2} $, and let $ \tau(\alpha) = (m, -w, o) $, where $ (m, w, o) \in \Phi(a, S_{4,9}(a)) $. Since $ w_1 = a+2 $ and $ w_2 = 0 $, by \Cref{lem:S4t1t2},
\begin{align*}
	t_1& = \frac{(a+1)w_2-(a+2)w_1}{a^2+3a+3}+2 = 2-\frac{(a+2)^2}{a^2+3a+3} \geq \frac{1}{2},\\
	t_2& = \frac{w_1}{a^2+3a+3} = \frac{a+2}{a^2+3a+3} \geq \frac{1}{2a}.
\end{align*}
We also have
\[
	t_3 = \frac{w}{a^2+3a+3} = w_0+\frac{(a+1)w_1}{a^2+3a+3}+\frac{w_2}{a^2+3a+3} \geq w_0+\frac{(a+1)(a+2)}{a^2+3a+3} \geq w_0+\frac{1}{2}.
\]

If $ \Nm(\alpha) \leq x $, then by \Cref{lem:Estimates}, $ a^4t_1t_2t_3 < x $, hence $ a^3\left(w_0+\frac{1}{2}\right) < 4x $. Let
\[
	M(a, x) := \left\{(w_0, w_1, w_2) \in \Z_{\geq 0}^3 : w_0 \geq 0, w_1 = a+2, w_2 = 0, a^3\left(w_0+\frac{1}{2}\right) < 4x\right\}.
\]
We have
\[
	\#\left\{\alpha \in Q(a, S_{4,9}(a)) : \Nm(\alpha) \leq x\right\} \leq \#M(a, x) \ll \frac{x}{a^3}.\qedhere
\]
\end{proof}

\begin{lemma}
\label{lem:S410}
Let $ a \in \Z $, $ a \geq 8 $. If $ x \geq 1 $, then
\[
	\#\left\{\alpha \in Q(a, S_{4,10}(a)) : \Nm(\alpha) \leq x\right\} \ll \frac{x}{a^2}(\log a)+\left(\frac{x}{a}\right)^{1/2}.
\]
\end{lemma}
\begin{proof}
Let $ \alpha \in Q(a, S_{4,10}(a)) $, let $ t_1, t_2, t_3 \in \R_{>0} $ be such that $ \alpha = t_1+t_2\rho^2+t_3(\rho')^{-2} $, and let $ \tau(\alpha) = (m, -w, o) $, where $ (m, w, o) \in \Phi(a, S_{4,10}(a)) $. Since $ 1 \leq w_1 \leq a+1 $ and $ w_2 = 0 $, by \Cref{lem:S4t1t2},
\begin{align*}
	t_1& = \frac{(a+1)w_2-(a+2)w_1}{a^2+3a+3}+1 = 1-\frac{(a+2)w_1}{a^2+3a+3},\\
	t_2& = \frac{w_1}{a^2+3a+3} \geq \frac{w_1}{2a^2}.
\end{align*}
We also have
\[
	t_3 = \frac{w}{a^2+3a+3} = w_0+\frac{(a+1)w_1}{a^2+3a+3}+\frac{w_2}{a^2+3a+3} \geq w_0.
\]

We consider two cases. First, assume $ w_1 \leq \frac{a}{2} $. Then
\[
	t_1 = 1-\frac{(a+2)w_1}{a^2+3a+3} \geq 1-\frac{(a+2)a}{2(a^2+3a+3)} \geq \frac{1}{2}.
\]

If $ \Nm(\alpha) \leq x $, then by \Cref{lem:Estimates}, $ a^4t_1t_2t_3 < x $, hence $ a^2w_0w_1 < 4x $. If we let
\[
	M_1(a, x) := \left\{(w_0, w_1, w_2) \in \Z_{\geq 0}^3 : w_0 \geq 1, 1 \leq w_1 \leq \frac{a}{2}, w_2 = 0, a^2w_0w_1 < 4x\right\},
\]
then
\[
	\#M_1(a, x) \leq \sum_{1 \leq w_1 \leq \frac{a}{2}}\frac{4x}{a^2w_1} \ll \frac{x}{a^2}\sum_{1 \leq w_1 \leq \frac{a}{2}}\frac{1}{w_1} \ll \frac{x}{a^2}(\log a).
\]

Secondly, assume $ w_1 > \frac{a}{2} $. Then
\[
	t_2 \geq \frac{w_1}{2a^2} > \frac{1}{4a}.
\]
We further distinguish two cases. First, let $ w_1 = a+1 $. If $ \Nm(\alpha) \leq x $, then $ a^2t_2t_3^2 < 2x $ by \Cref{lem:Estimates}, hence $ aw_0^2 < 8x $. If we let
\[
	M_2(a, x) := \left\{(w_0, w_1, w_2) \in \Z_{\geq 0}^3 : w_0 \geq 1, w_1 = a+1, w_2 = 0, aw_0^2 < 8x\right\},
\]
then $ \#M_2(a, x) \ll \left(\frac{x}{a}\right)^{1/2} $.

If $ w_1 < a+1 $, then we let $ w_1 = (a+1)-z_1 $, where $ 1 \leq z_1 < \frac{a}{2}+1 $. Then
\[
	t_1 = 1-\frac{(a+2)((a+1)-z_1)}{a^2+3a+3} = 1-\frac{(a+2)(a+1)}{a^2+3a+3}+\frac{(a+2)z_1}{a^2+3a+3} \geq \frac{(a+2)z_1}{a^2+3a+3} \geq \frac{z_1}{2a}.
\]

If $ \Nm(\alpha) \leq x $, then by \Cref{lem:Estimates}, $ a^4t_1t_2t_3 < x $, hence $ a^2w_0z_1 < 8x $. If we let
\[
	M_3(a, x) := \left\{(w_0, w_1, w_2) \in \Z_{\geq 0}^3 : w_0 \geq 1, w_1 = (a+1)-z_1, 1 \leq z_1 < \frac{a}{2}+1, w_2 = 0, a^2w_0z_1 < 8x\right\},
\]
then
\[
	\#M_3(a, x) \leq \sum_{1 \leq z_1 < \frac{a}{2}+1}\frac{8x}{a^2z_1} \ll \frac{x}{a^2}\sum_{1 \leq z_1 < \frac{a}{2}+1}\frac{1}{z_1} \ll \frac{x}{a^2}(\log a).
\]

Finally,
\[
	\#\left\{\alpha \in Q(a, S_{4,10}(a)) : \Nm(\alpha) \leq x\right\} \leq \#M_1(a, x)+\#M_2(a, x)+\#M_3(a, x) \ll \frac{x}{a^2}(\log a)+\left(\frac{x}{a}\right)^{1/2}.\qedhere
\]
\end{proof}

\begin{lemma}
\label{lem:S411}
Let $ a \in \Z $, $ a \geq 8 $. If $ x \geq 1 $, then
\[
	\#\left\{\alpha \in Q(a, S_{4,11}(a)) : \Nm(\alpha) \leq x\right\} \ll \left(\frac{x}{a}\right)^{1/2}+\frac{x}{a^2}.
\]
\end{lemma}
\begin{proof}
Let $ \alpha \in Q(a, S_{4,11}(a)) $, let $ t_1, t_2, t_3 \in \R_{>0} $ be such that $ \alpha = t_1+t_2\rho^2+t_3(\rho')^{-2} $, and let $ \tau(\alpha) = (m, -w, o) $, where $ (m, w, o) \in \Phi(a, S_{4,11}(a)) $. Since $ 1 \leq w_1 \leq a+1 $ and $ w_2 = 0 $, by \Cref{lem:S4t1t2},
\begin{align*}
	t_1& = \frac{(a+1)w_2-(a+2)w_1}{a^2+3a+3}+1 = 1-\frac{(a+2)w_1}{a^2+3a+3},\\
	t_2& = \frac{w_1}{a^2+3a+3} \geq \frac{w_1}{2a^2}.
\end{align*}
We also have
\[
	t_3 = \frac{w}{a^2+3a+3} = w_0+\frac{(a+1)w_1}{a^2+3a+3}+\frac{w_2}{a^2+3a+3} \geq \frac{(a+1)w_1}{a^2+3a+3} \geq \frac{w_1}{2a}.
\]

We consider two cases. First, assume $ w_1 \leq \frac{a}{2} $. Then
\[
	t_1 = 1-\frac{(a+2)w_1}{a^2+3a+3} \geq 1-\frac{(a+2)a}{2(a^2+3a+3)} \geq \frac{1}{2}.
\]

If $ \Nm(\alpha) \leq x $, then by \Cref{lem:Estimates}, $ a^4t_1t_2t_3 < x $, hence $ aw_1^2 < 8x $. If we let
\[
	M_1(a, x) := \left\{(w_0, w_1, w_2) \in \Z_{\geq 0}^3 : w_0 = 0, 1\leq w_1 \leq \frac{a}{2}, w_2 = 0, aw_1^2 < 8x\right\},
\]
then $ \#M_1(a, x) \ll \left(\frac{x}{a}\right)^{1/2} $.

Secondly, assume $ w_1 > \frac{a}{2} $. Then $ t_2 \geq \frac{w_1}{2a^2} > \frac{1}{4a} $ and $ t_3 \geq \frac{w_1}{2a} > \frac{1}{4} $.

We further distinguish two cases. If $ w_1 = a+1 $, then
\[
	t_1 = 1-\frac{(a+2)(a+1)}{a^2+3a+3} = \frac{1}{a^2+3a+3} \geq \frac{1}{2a^2}.
\]
If $ \Nm(\alpha) \leq x $, then by \Cref{lem:Estimates}, $ a^4t_1t_2t_3 < x $, hence $ a<32x $ and $ 1 \ll \left(\frac{x}{a}\right)^{1/2} $. If we let
\[
	M_2(a,x) := \left\{(w_0,w_1,w_2)\in \Z_{\geq 0}^3 : w_0 = 0, w_1 = a+1, w_2 = 0, a<32x\right\},
\]
then $ \#M_2(a,x) $ equals $ 0 $ or $ 1 $, and in both cases $ \#M_2(a,x) \ll \left(\frac{x}{a}\right)^{1/2} $.

If $ w_1 < a+1 $, then we let $ w_1 = (a+1)-z_1 $, where $ 1 \leq z_1 < \frac{a}{2}+1 $. We have
\[
	t_1 = 1-\frac{(a+2)w_1}{a^2+3a+3} = 1-\frac{(a+2)((a+1)-z_1)}{a^2+3a+3} = \frac{(a+2)z_1}{a^2+3a+3} \geq \frac{z_1}{2a}.
\]

If $ \Nm(\alpha) \leq x $, then by \Cref{lem:Estimates}, $ a^4t_1t_2t_3 < x $, hence $ a^2z_1 < 32x $. If we let
\[
	M_3(a, x) := \left\{(w_0, w_1, w_2)\in \Z_{\geq 0}^3 : w_0 = 0, w_1 = (a+1)-z_1, 1 \leq z_1 < \frac{a}{2}+1, w_2 = 0, a^2z_1 < 32x\right\},
\]
then $ \#M_3(a, x) \ll \frac{x}{a^2} $.

Finally,
\[
	\#\left\{\alpha \in Q(a, S_{4,11}(a)) : \Nm(\alpha) \leq x\right\} \ll \#M_1(a, x)+\#M_2(a,x)+\#M_3(a,x) \ll \left(\frac{x}{a}\right)^{1/2}+\frac{x}{a^2}.\qedhere
\]
\end{proof}

\begin{lemma}
\label{lem:S412}
Let $ a \in \Z $, $ a \geq 8 $. If $ x \geq 1 $, then
\[
	\#\left\{\alpha \in Q(a, S_{4,12}(a)) : \Nm(\alpha) \leq x\right\} \ll \frac{x}{a^4}.
\]
\end{lemma}
\begin{proof}
Let $ \alpha \in Q(a, S_{4,12}(a)) $, let $ t_1, t_2, t_3 \in \R_{>0} $ be such that $ \alpha = t_1+t_2\rho^2+t_3(\rho')^{-2} $, and let $ \tau(\alpha) = (m, -w, o) $, where $ (m, w, o) \in \Phi(a, S_{4,12}(a)) $. Since $ w_1 = 0 $ and $ w_2 = 0 $, by \Cref{lem:S4t1t2},
\begin{align*}
	t_1& = \frac{(a+1)w_2-(a+2)w_1}{a^2+3a+3}+1 = 1,\\
	t_2& = 1.
\end{align*}
We also have
\[
	t_3 = \frac{w}{a^2+3a+3} = w_0+\frac{(a+1)w_1}{a^2+3a+3}+\frac{w_2}{a^2+3a+3} \geq w_0.
\]

If $ \Nm(\alpha) \leq x $, then by \Cref{lem:Estimates}, $ a^4t_1t_2t_3 < x $, hence $ a^4w_0 < x $. Thus,
\[
	\#\left\{\alpha \in Q(a, S_{4,12}(a)) : \Nm(\alpha) \leq x\right\} \ll \frac{x}{a^4}.\qedhere
\]
\end{proof}

\begin{prop}
\label{prop:S4}
Let $ a \in \Z $, $ a \geq 8 $. If $ x \geq 1 $, then
\[
	\#\left\{\alpha \in Q(a, S_4(a)) : \Nm(\alpha) \leq x\right\} \ll \frac{x}{a^2}(\log a)^2+\left(\frac{x}{a}\right)^{2/3}.
\]
\end{prop}
\begin{proof}
We have $ S_4(a) = S_{4,1}(a)\sqcup S_{4,2}(a)\sqcup \dots \sqcup S_{4,12}(a) $, hence
\[
	Q(a, S_4(a)) = Q(a, S_{4,1}(a))\sqcup Q(a, S_{4,2}(a))\sqcup \dots\sqcup Q(a, S_{4,12}(a)).
\]
The proposition now follows from \Crefrange{lem:S41}{lem:S412}.
\end{proof}

\subsection{Proof of the upper bound}

We recall from \Cref{sec:Setup} that
\[
	P(a,x) = \left\{\alpha \in \O_{K_a}^+ : \Nm(\alpha) \leq x, \tau(\alpha) \in \calQ\right\},
\]
where
\[
	\calQ = \calC(1, \rho^2, (\rho'')^{-2})\sqcup \calC(1, \rho^2, (\rho')^{-2})\sqcup \calC(1, \rho^2)\sqcup \calC(1, (\rho'')^{-2})\sqcup \calC(1, (\rho')^{-2})\sqcup \calC(1).
\]

\begin{prop}
\label{prop:C3}
Let $ a \in \Z $, $ a \geq 8 $ and $ r > 0 $. If $ 1 \leq x \leq a^r $, then
\[
	P\left(a,x,\calC(1,\rho^2,(\rho')^{-2})\right) \ll_r \frac{x}{a^2}(\log a)^2+\left(\frac{x}{a}\right)^{2/3}.
\]
The constant implied in the estimate depends only on $ r $.
\end{prop}
\begin{proof}
If we let $ \calC = \calC(1,\rho^2,(\rho')^{-2}) $, then
\[
	\left\{\alpha \in \O_{K_a}^+:\tau(\alpha) \in \calC\right\} = Q(a,S_1(a))\sqcup Q(a,S_2(a))\sqcup R(a,S_3(a))\sqcup Q(a,S_4(a)).	
\]
By \Cref{prop:S1}, \Cref{prop:S2}, \Cref{prop:S3}, and \Cref{prop:S4},
\begin{align*}
	\#\left\{\alpha \in Q(a,S_1(a)):\Nm(\alpha)\leq x\right\}& \ll_r \frac{x}{a^2}(\log a)^2,\\
	\#\left\{\alpha \in Q(a,S_2(a)):\Nm(\alpha)\leq x\right\}& \ll_r \frac{x}{a^2}(\log a)^2+\left(\frac{x}{a}\right)^{2/3},\\
	\#\left\{\alpha \in R(a,S_3(a)):\Nm(\alpha)\leq x\right\}& \ll_r \frac{x}{a^2}(\log a)^2+\left(\frac{x}{a}\right)^{2/3},\\
	\#\left\{\alpha \in Q(a,S_4(a)):\Nm(\alpha)\leq x\right\}& \ll_r \frac{x}{a^2}(\log a)^2+\left(\frac{x}{a}\right)^{2/3}.
\end{align*}
The result follows.
\end{proof}

\begin{theorem}
\label{thm:UB}
Let $ K_a = \Q(\rho) $ be a simplest cubic field, where $ \rho $ is the largest root of $ f_a $. Assume that $ \O_{K_a} = \Z[\rho] $. If $ a \geq 8 $ and $ x \geq 1 $, then
\[
	P(a, x) \ll x^{1/3}+\left(\frac{x}{a}\right)^{2/3}+\frac{(\log a)^2x}{a^2}
\]
and
\[
	P_p(a,x) \ll 1+\left(\frac{x}{a}\right)^{2/3}+\frac{(\log a)^2x}{a^2}.
\]
\end{theorem}
\begin{proof}
We have $ P_p(a,x) \leq P(a,x) $. From \Cref{thm:Apr}, we know that there exists a constant $ c > 0 $ such that if $ x > ca^6(\log a)^4 $, then
\[
	P(a,x) \ll \frac{(\log a)^2x}{a^2}.
\]
Thus, we may assume $ x \leq c a^6(\log a)^4 $ for the rest of the proof.

By \Cref{lem:C1},
\[
	P(a,x,\calC(1)) \ll x^{1/3}
\]
and if $ \calC \in \left\{\calC\left(1,\rho^2\right),\calC\left(1,(\rho'')^{-2}\right),\calC\left(1,(\rho')^{-2}\right)\right\} $, then
\[
	P(a,x,\calC) \ll \left(\frac{x}{a^2}\right)^{2/3} \leq \left(\frac{x}{a}\right)^{2/3}.
\]
Moreover, the only primitive element in $ \calC(1) $ is $ \alpha = 1 $, hence
\[
	P_p(a,x,\calC(1)) \ll 1.
\]

By \Cref{lem:C2},
\[
	P\left(a,x,\calC\left(1,\rho^2,(\rho'')^{-2}\right)\right) \ll \frac{(\log a)^2 x}{a^2},
\]
and by \Cref{prop:C3},
\[
	P\left(a,x,\calC\left(1,\rho^2,(\rho')^{-2}\right)\right) \ll \frac{(\log a)^2x}{a^2}+\left(\frac{x}{a}\right)^{2/3}.
\]
The theorem follows by summing the above estimates.
\end{proof}

\section{Lower bound and proof of the main theorem}
\label{sec:LB}

In this section, we prove a lower bound for the number of primitive principal ideals $ P_p(a,x) $. For an element $ \alpha \in \O_{K_a} $ expressed in the basis $ (1,\rho,\rho^2) $ as $ \alpha = m+n\rho+o\rho^2 $, i.e., $ \tau(\alpha) = (m,n,o) $, we have that $ \alpha $ is primitive if and only if $ \gcd(m,n,o) = 1 $. Let us introduce the notation
\[
	\Prim_a := \left\{\alpha \in \O_{K_a} : n \nmid \alpha, n \in \Z_{\geq 2}\right\}
\]
for the set of primitive integral elements in $ K_a $.

\begin{lemma}
\label{lem:NU}
Let $ a \in \Z $, $ a \geq 8 $. Let $ t_1, t_2, t_3 \in \R_{>0} $ and $ \alpha = t_1+t_2\rho^2+t_3(\rho')^{-2} $. We have
\[
	\Nm(\alpha) < \left(t_1+2a^2t_2+t_3\right)\left(t_1+2t_2+2a^2t_3\right)\left(t_1+\frac{t_2}{a^2}+\frac{t_3}{a^2}\right).
\]
\end{lemma}
\begin{proof}
The norm of $ \alpha $ equals
\[
	\Nm(\alpha) = \left(t_1+t_2\rho^2+t_3(\rho')^{-2}\right)\left(t_1+t_2(\rho')^2+t_3(\rho'')^{-2}\right)\left(t_1+t_2(\rho'')^2+t_3\rho^{-2}\right).
\]
By \Cref{lem:Units}, the sizes of the units can be estimated as $ \rho^2 < 2a^2 $, $ (\rho')^{-2} < 1 $, $ (\rho')^2 < 2 $, $ (\rho'')^{-2} < 2a^2 $, $ (\rho'')^2 < \frac{1}{a^2} $, and $ \rho^{-2} < \frac{1}{a^2} $. The inequality follows.
\end{proof}

As \Cref{lem:S25} is the first case where we obtained the upper bound $ \left(\frac{x}{a}\right)^{2/3} $, we will now prove that it is the correct order of magnitude. We recall that
\[
	S_2(a) := \left\{(m_1,o_1,w_0,w_1,w_2)\in \Z^2\times W_1(a) : m_1 \geq 2, o_1 = 1\right\}
\]
and
\[
	S_{2,5}(a) := \left\{(m_1,o_1,w_0,w_1,w_2) \in S_2(a) : w_0 = 0, w_1 \geq 1, w_2 = 0\right\}.
\]

\begin{lemma}
\label{lem:S25L}
Let $ a \in \Z $, $ a \geq 8 $. If $ x \geq 5^3a $, then
\[
	\#\left\{\alpha \in Q(a, S_{2,5}(a)) \cap \Prim_a : \Nm(\alpha) \leq x\right\} \gg \left(\frac{x}{a}\right)^{2/3}.
\]
\end{lemma}
\begin{proof}
Let $ t_1, t_2, t_3 \in \R_{>0} $, $ \alpha = t_1+t_2\rho^2+t_3(\rho')^{-2} $, and let $ \tau(\alpha) = (m,-w,o) $.

We have $ m = m_0(a,w)+m_1 $ and $ o = o_0(a,w)+o_1 $. If we assume that $ m \geq 1 $, $ o_1 = 1 $, and $ (w_0,w_1,w_2) = \phi_a(w) $, where $ w_0 = 0 $, $ w_1 \geq 1 $, and $ w_2 = 0 $, then $ \alpha \in Q(a,S_{2,5}(a)) $. Indeed, $ m_0(a,w) = \left\lfloor-\frac{(a+1)w}{a^2+3a+3}\right\rfloor < 0 $, hence $ m_1 \geq 2 $.

By \Cref{lem:W2},
\[
	o = o_0(a,w)+1 = \left\lfloor\frac{a+2}{a^2+3a+3}w\right\rfloor+1 = (a+2)w_0+w_1-1+1 = w_1.
\]
Thus, if $ \gcd(m,w_1) = 1 $, then $ \alpha $ is primitive.

We have
\begin{align*}
	t_1& = m+\frac{a+1}{a^2+3a+3}w = m+\frac{(a+1)^2w_1}{a^2+3a+3} \leq m+w_1,\\
	t_2& = \frac{w_1}{a^2+3a+3} \leq \frac{w_1}{a^2},\\
	t_3& = \frac{w}{a^2+3a+3} = \frac{(a+1)w_1}{a^2+3a+3} \leq \frac{w_1}{a},
\end{align*}
where we used \Cref{lem:Estimatet2} to estimate $ t_2 $. By \Cref{lem:NU},
\begin{align*}
	\Nm(\alpha)& < \left(m+3w_1+\frac{w_1}{a}\right)\left(m+w_1+\frac{2w_1}{a^2}+2aw_1\right)\left(m+w_1+\frac{w_1}{a^4}+\frac{w_1}{a^3}\right)\\
	& < \left(m+4w_1\right)\left(m+(2a+2)w_1\right)\left(m+3w_1\right) < a\left(m+4w_1\right)^3.
\end{align*}
Thus, if $ 1 \leq m \leq \frac{1}{5}\left(\frac{x}{a}\right)^{1/3} $ and $ 1 \leq w_1 \leq \frac{1}{5}\left(\frac{x}{a}\right)^{1/3} $, then $ \Nm(\alpha) < x $. If we let
\[
	M(a, x) := \left\{(m,w_1) \in \Z_{\geq 1}^2 : \gcd(m,w_1) = 1, 1 \leq m \leq \frac{1}{5}\left(\frac{x}{a}\right)^{1/3}, 1 \leq w_1 \leq \frac{1}{5}\left(\frac{x}{a}\right)^{1/3}\right\}
\]
and $ y = \frac{1}{5}\left(\frac{x}{a}\right)^{1/3} $, then
\[
	\#M(a,x) = \sum_{1 \leq w_1 \leq y}\sum_{\substack{1 \leq m \leq y\\\gcd(m,w_1) = 1}}1 \gg y^2 \gg \left(\frac{x}{a}\right)^{2/3},
\]
where we used $ y \geq 1 $.

Finally,
\[
	\#\left\{\alpha \in Q(a,S_{2,5}(a)) \cap \Prim_a : \Nm(\alpha) \leq x\right\} \geq \#M(a,x) \gg \left(\frac{x}{a}\right)^{2/3}.\qedhere
\]
\end{proof}

We also prove a lower bound for the number of primitive elements $ \alpha \in Q(a,S_1(a)) $. We recall that
\[
	S_1(a) = \left\{(m_1,o_1,w_0,w_1,w_2)\in \Z^2\times W_1(a) : m_1 \geq 2, o_1 \geq 2\right\}.
\]

\begin{lemma}
\label{lem:S1L}
Let $ a \in \Z $, $ a \geq 8 $. There exists a constant $ c_1 > 0 $ such that if $ x \geq c_1a^3 $, then
\[
	\#\left\{\alpha \in Q(a,S_1(a)) \cap \Prim_a : \Nm(\alpha) \leq x\right\} \gg \frac{(\log a)^2x}{a^2}.
\]
\end{lemma}
\begin{proof}
Let $ t_1, t_2, t_3 \in \R_{>0} $, $ \alpha = t_1+t_2\rho^2+t_3(\rho')^{-2} $, and let $ \tau(\alpha) = (m,-w,o) $. Let $ m = m_0(a,w)+m_1 $ and $ o = o_0(a,w)+o_1 $. We assume $ m \geq 1 $ and $ o_1 \geq 2 $. Since $ m_0(a,w) < 0 $, this implies $ m_1 \geq 2 $ and $ \alpha \in Q(a,S_1(a)) $. We have
\begin{align*}
	t_1& = m+\frac{a+1}{a^2+3a+3}w \leq m+\frac{w}{a},\\
	t_2& \leq o_1,\\
	t_3& = \frac{w}{a^2+3a+3} \leq \frac{w}{a^2},
\end{align*}
where we used  \Cref{lem:t1t2} to estimate $ t_2 $. By \Cref{lem:NU},
\begin{align*}
	\Nm(\alpha)& < \left(m+\frac{w}{a}+2a^2o_1+\frac{w}{a^2}\right)\left(m+\frac{w}{a}+2o_1+2w\right)\left(m+\frac{w}{a}+\frac{o_1}{a^2}+\frac{w}{a^4}\right)\\
	& \leq \left(m+\frac{2w}{a}+2a^2o_1\right)\left(m+3w+2o_1\right)\left(m+\frac{2w}{a}+\frac{o_1}{a^2}\right) = m^3+\frac{12w^3}{a^2}+4o_1^3+\left(3+\frac{4}{a}\right)m^2w\\
	&+\left(\frac{12}{a}+\frac{4}{a^2}\right)mw^2+\left(2a^2+2+\frac{1}{a^2}\right)m^2o_1+\left(4a^2+2+\frac{2}{a^2}\right)mo_1^2+\left(12a+\frac{8}{a^2}+\frac{6}{a^3}\right)w^2o_1\\
	&+\left(8a+6+\frac{4}{a^3}\right)wo_1^2+\left(6a^2+4a+\frac{8}{a}+\frac{3}{a^2}+\frac{2}{a^3}\right)mwo_1.
\end{align*}
Let $ c_2, c_3, c_4 > 0 $ be constants such that $ c_2 < c_3 $ and assume that
\[
	c_2\left(\frac{x}{a^2}\right)^{1/3} \leq m \leq c_3x^{1/3},\qquad c_2\left(\frac{x}{a^2}\right)^{1/3} \leq w \leq c_3x^{1/3},\qquad m \leq w,\qquad o_1 \leq c_4\frac{x}{a^2mw}.
\]
We get $ m^3 \leq c_3^3x $, $ w^3 \leq c_3^3x $, $ m^2w \leq c_3^3x $, $ mw^2 \leq c_3^3x $, and
\begin{align*}
	o_1^3& \leq c_4^3\left(\frac{x}{a^2mw}\right)^3 \leq \left(\frac{c_4}{c_2^2}\right)^3\frac{x}{a^2},\\
	a^2m^2o_1& \leq a^2m^2c_4\frac{x}{a^2mw} \leq c_4x\frac{m}{w} \leq c_4x,\\
	a^2mo_1^2& \leq a^2mc_4^2\frac{x^2}{a^4m^2w^2} \leq c_4^2\frac{x^2}{a^2mw^2} \leq \frac{c_4^2}{c_2^3}x,\\
	aw^2o_1& \leq aw^2c_4\frac{x}{a^2mw} = c_4\frac{x}{a}\frac{w}{m} \leq c_4\frac{x}{a}\frac{c_3x^{1/3}}{c_2(x/a^2)^{1/3}} \leq c_4\frac{c_3}{c_2}\frac{x}{a^{1/3}},\\
	awo_1^2& \leq awc_4^2\frac{x^2}{a^4m^2w^2} \leq c_4^2\frac{x^2}{a^3m^2w} \leq \frac{c_4^2}{c_2^3}\frac{x}{a},\\
	a^2mwo_1& \leq a^2mwc_4\frac{x}{a^2mw} \leq c_4x.
\end{align*}
Because all these expressions are smaller than a constant times $ x $, we can choose $ c_3 $, $ c_2 < c_3 $, and $ c_4 $ such that $ \Nm(\alpha) < x $. We take $ c_5 = 2^3c_2^{-3} $, so that for $ x \geq c_5a^2 $, we have $ 2 \leq c_2\left(\frac{x}{a^2}\right)^{1/3} $. Next we let $ 0 < r < 2 $ be a constant (which will be specified later) and define
\[
	M_1(a,x) := \left\{(m,w,o_1)\in \Z^3 : \gcd(m,w) = 1, m \leq w, c_2\left(\frac{x}{a^2}\right)^{1/3} \leq m, w \leq c_3\left(\frac{x}{a^r}\right)^{1/3}, 2 \leq o_1 \leq c_4\frac{x}{a^2mw}\right\}.
\]
If $ (m,w,o_1) \in M_1(a,x) $ and $ \alpha \in \O_{K_a}^+ $ is such that $ \tau(\alpha) = (m,-w,o) $, then $ \alpha \in Q(a,S_1(a)) \cap \Prim_a $ and $ \Nm(\alpha) < x $. Removing the condition $ m \leq w $, we get $ \#M_1(a,x) = \frac{1}{2}\#M(a,x) $, where
\[
	M(a,x) := \left\{(m,w,o_1)\in \Z^3 : \gcd(m,w) = 1, c_2\left(\frac{x}{a^2}\right)^{1/3} \leq m, w \leq c_3\left(\frac{x}{a^r}\right)^{1/3}, 2 \leq o_1 \leq c_4\frac{x}{a^2mw}\right\}.
\]
We have
\[
	\#M(a,x) = \sum_{\substack{c_2\left(\frac{x}{a^2}\right)^{1/3} \leq m \leq c_3\left(\frac{x}{a^r}\right)^{1/3},\\c_2\left(\frac{x}{a^2}\right)^{1/3} \leq w \leq c_3\left(\frac{x}{a^r}\right)^{1/3}\\\gcd(m,w)=1}}\left\lfloor c_4\frac{x}{a^2mw}\right\rfloor-1 \geq \left(\sum_{\substack{c_2\left(\frac{x}{a^2}\right)^{1/3} \leq m \leq c_3\left(\frac{x}{a^r}\right)^{1/3},\\c_2\left(\frac{x}{a^2}\right)^{1/3} \leq w \leq c_3\left(\frac{x}{a^r}\right)^{1/3}\\\gcd(m,w)=1}}c_4\frac{x}{a^2mw}\right)-2c_3^2\left(\frac{x}{a^r}\right)^{2/3}.
\]
There exists a constant $ c_6 > 0 $ such that
\[
	\#M(a,x) \geq c_6\frac{x}{a^2}\log^2\left(\frac{c_3\left(\frac{x}{a^r}\right)^{1/3}}{c_2\left(\frac{x}{a^2}\right)^{1/3}}\right)-2c_3^2\left(\frac{x}{a^r}\right)^{2/3} = c_6\frac{x}{a^2}\left(\log\left(\frac{c_3}{c_2}\right)+\frac{2-r}{3}(\log a)\right)^2-2c_3^2\left(\frac{x}{a^r}\right)^{2/3}.
\]
We choose $ r = \frac{3}{2} $. There exists a constant $ c_7 > 0 $ such that
\[
	\#M(a,x) \geq c_7\frac{x}{a^2}(\log a)^2-2c_3^2\left(\frac{x}{a^{3/2}}\right)^{2/3} = c_7\frac{x}{a^2}(\log a)^2-2c_3^2\frac{x^{2/3}}{a}.
\]
We set $ c_8 = \left(\frac{4c_3^2}{c_7}\right)^3 $, so that for $ x \geq c_8\frac{a^3}{(\log a)^6} $ we have $ \frac{c_7}{2}\frac{x}{a^2}(\log a)^2 \geq 2c_3^2\frac{x^{2/3}}{a} $. In particular, if we let $ c_1 = \max\{c_5,c_8\} $, then $ x \geq c_1a^3 $ satisfies both $ x \geq c_5a^2 $ and $ x \geq c_8\frac{a^3}{(\log a)^6} $, hence
\[
	\#M(a,x) \geq \frac{c_7}{2}\frac{x}{a^2}(\log a)^2 \gg \frac{x}{a^2}(\log a)^2.
\]

Finally,
\[
	\#\left\{\alpha \in Q(a, S_1(a)) \cap \Prim_a : \Nm(\alpha) \leq x\right\} \geq \#M_1(a,x) = \frac{1}{2}\#M(a,x) \gg \frac{(\log a)^2x}{a^2}.\qedhere
\]
\end{proof}

We are ready to prove our main result, \Cref{thm:P}, which we restate here for convenience.

\begin{theorem}
\label{thm:P'}
Let $ K_a = \Q(\rho) $ be a simplest cubic field, where $ \rho $ is the largest root of $ f_a $. Assume that $ \O_{K_a} = \Z[\rho] $. If $ a \geq 8 $ and $ x \geq 1 $, then
\[
	P(a, x) \asymp \frac{(\log a)^2x}{a^2}+\left(\frac{x}{a}\right)^{2/3}+x^{1/3}
\]
and
\[
	P_p(a, x) \asymp \frac{(\log a)^2x}{a^2}+\left(\frac{x}{a}\right)^{2/3}+1.
\]
\end{theorem}
\begin{proof}
We have $ P_p(a,x) \leq P(a,x) $. The upper bound follows from \Cref{thm:UB}. By \Cref{lem:C1},
\[
	P\left(a,x,\calC(1)\right) \gg x^{1/3},
\]
hence $ P(a,x) \gg x^{1/3} $. It remains to prove the lower bound for $ P_p(a,x) $.

If $ x \geq 5^3a $, then by \Cref{lem:S25L},
\[
	P_p(a,x) \geq \#\left\{\alpha \in S_{2,5}(a) \cap \Prim_a : \Nm(\alpha) \leq x\right\} \gg \left(\frac{x}{a}\right)^{2/3}.
\]
The ideal $ I = \O_{K_a} $ is primitive and $ \Nm(I) = 1 $. Thus, if $ 1 \leq x < 5^3a $, then $ P_p(a,x) \geq 1 \gg \left(\frac{x}{a}\right)^{2/3} $.

Finally, we get from \Cref{lem:S1L} that there exists a constant $ c_1 > 0 $ such that if $ x \geq c_1a^3 $, then
\[
	P_p(a, x) \geq \#\left\{\alpha \in S_1(a) \cap \Prim_a : \Nm(\alpha) \leq x\right\} \gg \frac{x(\log a)^2}{a^2}.
\]
But if $ 1 \leq x < c_1a^3 $, then
\[
	P_p(a,x) \gg \left(\frac{x}{a}\right)^{2/3} \gg \frac{x}{a^2}(\log a)^2. \qedhere
\]
\end{proof}


\begin{thebibliography}{BK15b}
\bibitem[Ba16]{Ba16}
S. Balady, \textit{Families of cyclic cubic fields}, J. Number Theory \textbf{167} (2016), 394--406.
\bibitem[Ber12]{Ber12}
P. Bernays, \textit{\"{U}ber die Darstellung von positiven, ganzen Zahlen durch die primitiven bin\"{a}ren quadratischen Formen einer nichtquadratischen Discriminante}, Dissertation, G\"{o}ttingen, 1912.
\bibitem[BG06]{BG}
V. Blomer and A. Granville, \textit{Estimates for representation numbers of quadratic forms}, Duke Math. J. \textbf{135} (2006), no. 2, 261--302.
\bibitem[BK15a]{BK1}
V. Blomer and V. Kala, \textit{Number fields without $ n $-ary universal quadratic forms}, Math. Proc. Cambridge Philos. Soc. \textbf{159} (2015), no. 2, 239--252.
\bibitem[BK15b]{BK2}
V. Blomer and V. Kala, \textit{On the rank of universal quadratic forms over real quadratic fields}, Doc. Math. \textbf{23} (2015), 15--34.
\bibitem[By00]{By00}
D. Byeon, \textit{Class number $3$ problem for the simplest cubic fields}, Proc. Amer. Math. Soc. \textbf{128} (2000), no. 5, 1319--1323.
\bibitem[De19]{De19}
K. Debaene, \textit{Explicit counting of ideals and a Brun-Titschmarsh inequality for the Chebotarev density theorem}, Int. J. Number Theory \textbf{15} (2019), no. 5, 883--905.
\bibitem[Fo11]{Fo11}
K. Foster, \textit{HT90 and ``simplest" number fields}, Illinois J. Math. \textbf{55} (2011), no. 4, 1621--1655.
\bibitem[GT22]{GT22}
D. Gil Mu\~{n}oz and M. Tinkov\'{a}, \textit{Additive structure of non-monogenic simplest cubic fields}, \url{https://arxiv.org/abs/2212.00364}, 2022.
\bibitem[Go60]{Go60}
H. J. Godwin, \textit{The determination of units in totally real cubic fields}, Proc. Cambridge Philos. Soc. \textbf{56} (1960), 318--321.
\bibitem[KT22]{KT}
V. Kala and M. Tinková, \textit{Universal quadratic forms, small norms, and traces in families of number fields}, Int. Math. Res. Not. IMRN 2023, no. 9, 7541--7577.
\bibitem[Ki00]{Ki00}
B. M. Kim, \textit{Universal octonary diagonal forms over some real quadratic fields}, Comment. Math. Helv. \textbf{75} (2000), no. 3, 410--414.
\bibitem[Ki03]{Ki03}
Y. Kishi, \textit{A family of cyclic cubic polynomials whose roots are systems of fundamental units}, J. Number Theory \textbf{102} (2003), no. 1, 90--106.
\bibitem[La49]{La49}
E. Landau, \textit{Einf\"{u}hrung in die elementare und analytische Theorie der algebraischen Zahlen und der Ideale}, Chelsea Publishing Co., New York, 1949. vii+147 pp.
\bibitem[Le22]{Le22}
E. S. Lee, \textit{On the number of integral ideals in a number field}, J. Math. Anal. Appl. \textbf{517} (2023), no. 1, Paper No. 126585, 25 pp.
\bibitem[LP95]{LP}
F. Lemmermeyer and A. Peth\"{o}, \textit{Simplest cubic fields}, Manuscripta Math. \textbf{88} (1995), no. 1, 53--58.
\bibitem[MO07]{MO}
M. R. Murty and J. Van Order, \textit{Counting integral ideals in a number field}, Expo. Math. \textbf{25} (2007), no. 1, 53--66.
\bibitem[Na04]{Na04}
W. Narkiewicz, \textit{Elementary and analytic theory of algebraic numbers}, Third edition, Springer Monogr. Math., Springer-Verlag, Berlin, 2004. xii+708 pp.
\bibitem[Ne99]{Ne99}
J. Neukirch, \textit{Algebraic number theory.} Translated from the 1992 German original and with a note by Norbert Schappacher. With a foreword by G. Harder, Grundlehren Math. Wiss., 322 [Fundamental Principles of Mathematical Sciences], Springer-Verlag, Berlin, 1999. xviii+571 pp.
\bibitem[Sh74]{Sh74}
D. Shanks, \textit{The simplest cubic fields}, Math. Comp. \textbf{28} (1974), 1137--1152.
\bibitem[Su71]{Su71}
J. E. S. Sunley, \textit{On the class numbers of totally imaginary quadratic extensions of totally real fields}, Thesis (Ph.D.)--University of Maryland, College Park, ProQuest LLC, Ann Arbor, MI, 1971. 111 pp.
\bibitem[Ti23a]{Ti23a}
M. Tinkov\'{a}, \textit{On the Pythagoras number of the simplest cubic fields}, Acta Arith. \textbf{208} (2023), no. 4, 325--354.
\bibitem[Ti23b]{Ti23b}
M. Tinkov\'{a}, \textit{Trace and norm of indecomposable integers in cubic orders}, Ramanujan J. \textbf{61} (2023), no. 4, 1121--1144.
\bibitem[Wa87]{Wa87}
L. C. Washington, \textit{Class numbers of the simplest cubic fields}, Math. Comp. \textbf{48} (1987), no. 177, 371--384.
\bibitem[We61]{We61}
H. Weber, \textit{Lehrbuch der Algebra, Zweiter Band}, Third edition, Chelsea Publishing Co., New York, 1961.
\end{thebibliography}
\end{document}